\newif\iffinal
\finalfalse	% Not final version
\documentclass[letterpaper,11pt,reqno]{amsart} 

\RequirePackage[utf8]{inputenc}
\usepackage[portrait,margin=3cm]{geometry}
\usepackage{color}              % Need the color package
\usepackage[usenames,dvipsnames]{xcolor}

\usepackage{mathrsfs,soul}
\usepackage{hyperref}
\usepackage{amssymb,amsthm,amsmath,amsfonts,amsbsy,latexsym,dsfont}
\usepackage[textsize=small]{todonotes}       % includes TO DO LIST
\usepackage{graphicx}
\usepackage[numeric,initials,nobysame]{amsrefs}
\usepackage{upref,setspace}

\definecolor{BrickRed}{rgb}{0.65,0.08,0}

\usepackage{enumerate}

\numberwithin{equation}{section}
\numberwithin{figure}{section}
\numberwithin{table}{section}

\newtheorem{Lemma}{Lemma}[section]
\newtheorem{Proposition}{Proposition}[section]

\newtheorem{Theorem}{Theorem}[section]

\newtheorem{Condition}{Condition}[section]

\newtheorem{Remark}{Remark}[section]
\newtheorem{Example}{Example}[section]

\newtheorem{Corollary}{Corollary}[section]

\newcommand{\lfl}{\lfloor}
\newcommand{\rfl}{\rfloor}

\newcommand{\Rd}{{\Rmb^d}}

\newcommand{\Cmb}{{\mathbb{C}}}

\newcommand{\Emb}{{\mathbb{E}}}

\newcommand{\Nmb}{{\mathbb{N}}}

\newcommand{\Pmb}{{\mathbb{P}}}

\newcommand{\Rmb}{{\mathbb{R}}}
\newcommand{\Smb}{{\mathbb{S}}}

\newcommand{\Bmc}{{\mathcal{B}}}

\newcommand{\Gmc}{{\mathcal{G}}}

\newcommand{\Lmc}{{\mathcal{L}}}

\newcommand{\Pmc}{{\mathcal{P}}}

\newcommand{\Rmc}{{\mathcal{R}}}
\newcommand{\Smc}{{\mathcal{S}}}

\newcommand{\one}{{\boldsymbol{1}}}

\newcommand{\mubar}{{\bar{\mu}}}

\newcommand{\nubar}{{\bar{\nu}}}

\newcommand{\Smcbar}{{\bar{\Smc}}}

\newcommand{\Ybar}{{\bar{Y}}}

\newcommand{\btil}{{\tilde{b}}}

\newcommand{\etatil}{{\tilde{\eta}}}

\newcommand{\mtil}{{\tilde{m}}}

\newcommand{\Xtil}{{\tilde{X}}}

\numberwithin{equation}{section}

\newcommand{\emppre}{{\nu}}
\newcommand{\lawpre}{{\mu}}
\newcommand{\jointpre}{{\theta}}
\newcommand{\emplimit}{{\bar{\nu}}}
\newcommand{\lawlimit}{{\bar{\mu}}}
\newcommand{\lawlimitmean}{{\hat{\mu}}}
\newcommand{\Xlimit}{{\bar{X}}}

\begin{document}

	\title[Stationarity for the graphon particle system]{Stationarity and uniform in time convergence for the graphon particle system}

	\makeatletter
	\@namedef{subjclassname@2020}{\textup{2020} Mathematics Subject Classification}
	\makeatother
%	These three lines are added to cover the 2020 AMS subject class. The current version of amsart package does not have the 2020 option. Once the package is updated, these three lines could be delted.

	\date{\today}
	 
	\author[Bayraktar]{Erhan Bayraktar}
	\address{Department of Mathematics, University of Michigan, 530 Church Street, Ann Arbor, MI 48109}
%		\thanks{E. Bayraktar was supported in part by }  
	\author[Wu]{Ruoyu Wu}
	\address{Department of Mathematics, Iowa State University, 411 Morrill Road, Ames, IA 50011} 
	\email{erhan@umich.edu, ruoyu@iastate.edu}
			 
	\begin{abstract}	
	We consider the long time behavior of heterogeneously interacting diffusive particle systems and their large population limit.
	The interaction is of mean field type with weights characterized by an underlying graphon.
	The limit is given by a graphon particle system consisting of independent but heterogeneous nonlinear diffusions whose probability distributions are fully coupled.
	Under suitable assumptions, including a certain convexity condition, we show the exponential ergodicity for both systems, establish the uniform-in-time law of large numbers for marginal distributions as the number of particles increases, and introduce the uniform-in-time Euler approximation.
	The precise rate of convergence of the Euler approximation is provided.
	\end{abstract}	
	
	\maketitle

	\tableofcontents

\section{Introduction}

In this work we study the long time behavior of graphon particle systems and the finite particle approximations.
The interaction is of mean-field type and characterized by a graphon $G$, which is a symmetric measurable function from $[0,1] \times [0,1]$ to $[0,1]$ (see e.g.\ \cite{Lovasz2012large} for the theory of graphons).
More precisely, denoting by $\Xlimit_u$ the state of the particle at $u \in [0,1]$,
\begin{align}
	\Xlimit_u(t) & = \Xlimit_u(0) + \int_0^t \left( f(\Xlimit_u(s)) + \int_0^1 \int_{\Rmb^d} b(\Xlimit_u(s),x)G(u,v) \,\lawlimit_{v,s}(dx)\,dv\right)ds \notag \\
	& \qquad + \sigma B_u(t), \quad u \in [0,1], \: t \ge 0, \label{eq:system}
\end{align}
where $\lawlimit_{v,s}$ is the probability distribution of the $\Rd$-valued random variable $\Xlimit_v(s)$ for each $v \in [0,1]$ and $s \ge 0$, $f$ are $b$ are suitable functions, $\sigma \in \Rmb^{d \times d}$ is a constant, $\{B_u : u \in I\}$ are $d$-dimensional standard Brownian motions, and $\{\Xlimit_u(0),B_u : u \in I\}$ are mutually independent.
We will also study the mean-field particle system with heterogeneous interactions given by 
\begin{align}
	X_i^n(t) & = \Xlimit_{\frac{i}{n}}(0) + \int_0^t \left( f(X_i^n(s)) + \frac{1}{n} \sum_{j=1}^n \xi_{ij}^n b(X_i^n(s),X_j^n(s)) \right) ds
	\notag \\
	& \qquad + \sigma B_{\frac{i}{n}}(t), \quad i \in \{1,\dotsc,n\}, \: t \ge 0,
	\label{eq:system-n}
\end{align}
and its Euler discretization.
Here $\{\xi_{ij}^n: 1 \le i \le j \le n\}$ is a collection of independent $[0,1]$-valued random variables sampled from a step graphon $G_n$ that converges to the graphon $G$ in the cut metric.

The study of mean-field \textit{heterogeneously} interacting particle systems on random graphs converging to a graphon emerged recently (\cite{BayraktarChakrabortyWu2020graphon,BetCoppiniNardi2020weakly,Lucon2020quenched,OliveiraReis2019interacting,Coppini2019long}).
There is also a growing number of applications of graphons in game theory; see e.g.\ \cite{Carmona2019stochastic,PariseOzdaglar2019graphon,CainesHuang2018graphon,GaoTchuendomCaines2020linear,VasalMishraVishwanath2020sequential,CainesHuang2020graphon} for the study of graphon mean field games in static and dynamic settings.
Among these, the only work on long time analysis is \cite{Coppini2019long}, which shows that the stochastic Kuramoto model defined on a sequence of graphs converging to a constant graphon behaves asymptotically as the mean-field limit (in general the manifold of McKean--Vlasov equations), up to an exponential time. 
More precisely, it is shown in \cite{Coppini2019long} that, in the subcritical regime, the trajectory of the empirical measure of $n$ oscillators ($[0,2\pi]$-valued diffusions), over $t \in [0,T_n]$ where $T_n = \exp(o(n))$, converges uniformly in probability to the same limit as in the Kuramoto model; while in the supercritical regime, with initial states close to a stable stationary solution, the trajectory of the empirical measure is uniformly close to the manifold of stable stationary solutions over $t \in [0,T_n]$.
This is in contrast to our case where the limiting system \eqref{eq:system} is heterogeneous and the stationary measure relies crucially on the underlying graphon (see Example \ref{eg:Gaussian}).

The study of classic mean-field \textit{homogeneously} interacting particle systems and the associated limiting system given by nonlinear processes, or equivalently, McKean--Vlasov equations, dates back to works of Boltzmann, Vlasov, McKean and others; see \cite{Sznitman1991,McKean1967propagation,Kolokoltsov2010} and references therein. 
Besides large population limits such as law of large numbers (LLN) and propagation of chaos (POC) on the finite time horizon, there have been an extensive collection of results on long time behaviors and Euler approximations for such systems (see e.g.\ \cite{BudhirajaFan2017,BudhirajaPalMajumder2015long,Veretennikov2006ergodic,BolleyGuillinVillani2007quantitative} and references therein) with suitable convexity assumptions.
In recent ten years, there has been a growing interest in the mean-field \textit{inhomogeneous} particle system, where the interaction between particles is governed by their own types and/or random graphs (see e.g.\ \cite{BayraktarWu2019mean,BarreDobsonOttobreZatorska2020fast,BhamidiBudhirajaWu2019weakly,Delattre2016,CoppiniDietertGiacomin2019law,BudhirajaMukherjeeWu2019supermarket,Delarue2017mean,LackerSoret2020case}) and the limiting system consists of countable McKean-Vlasov processes (as opposed to uncountable heterogeneous processes in graphon particle systems like \eqref{eq:system}).
Among these, the paper \cite{BarreDobsonOttobreZatorska2020fast} considers a collection of diffusions interacting through state-dependent fast evolving random graphs and shows a uniform-in-time averaging and LLN result. 
The interaction in \cite{BarreDobsonOttobreZatorska2020fast}, although not in the mean-field form, is close to be mean-field due to the averaging effect, and the limiting system is given by independent and identically distributed (i.i.d.) nonlinear diffusions.

The goal of this work is to study the long time behavior, including the stationary distribution and uniform-in-time convergence, of the graphon particle system \eqref{eq:system}, the approximating finite particle system \eqref{eq:system-n}, and its Euler discretization, under our standing assumptions stated in Section \ref{sec:model}, including convexity conditions \eqref{eq:c0} and \eqref{eq:kappa}.
We are in particular interested in the following two questions:
\begin{enumerate}[1.]
\item 
	Knowing that $G_n \to G$ as $n \to \infty$, does the long time behavior of $\{X_i^n : i=1,\dotsc,n\}$ approach that of $\{\Xlimit_u : u \in [0,1]\}$?
\item 
	Given a graphon $G$ and the associated $\{\Xlimit_u : u \in [0,1]\}$, could one choose a sequence of graphons $G_n$, finite particle systems $\{X_i^n : i=1,\dotsc,n\}$, and the Euler discretizations to approximate the long time behavior of $\{\Xlimit_u : u \in [0,1]\}$ such as the stationary distribution? 
	In order to control the approximation error, what would be the balance between the number of particles, the time to run processes, and the discretization step size?
\end{enumerate}
The first question is natural and its finite time analogue has been answered in the works \cite{BayraktarChakrabortyWu2020graphon,BetCoppiniNardi2020weakly,Lucon2020quenched,OliveiraReis2019interacting,Coppini2019long} mentioned above with (possibly) different model setups. 
The second question, opposite to the first one, is also (and actually more) important to us, as one may worry that the graphon particle system \eqref{eq:system}, consisting of uncountably many heterogeneous particles (or equivalently, their probability distributions), is not always tractable, for the either finite-time or long-time behavior (see Example \ref{eg:Gaussian}).
Answering the second question will suggest, for example to simulate the stationary distributions of \eqref{eq:system} using the Euler discretization of \eqref{eq:system-n}, and how to choose parameters in a balanced and efficient manner. 

Our first main result is the exponential ergodicity of the two systems \eqref{eq:system} and \eqref{eq:system-n}.
For the graphon particle system \eqref{eq:system}, we show that $\lawlimit_{u,t}$ converges to a limiting distribution as $t \to \infty$ for each $u \in [0,1]$ with exponentially small errors, and the limit is invariant with respect to the system evolution (Theorem \ref{thm:exponential-ergodicity}).
For the finite particle system \eqref{eq:system-n}, although the random vector $(X_i^n(t) : i=1,\dotsc,n)$ is not Markovian, it is Markovian conditioned on the interaction $\{\xi_{ij}^n : 1 \le i \le j \le n\}$.
Using this observation we prove the quenched (and hence annealed) exponential ergodicity and also show that the quenched limiting distribution is invariant (Theorem \ref{thm:exponential-ergodicity-n}).

The second main result is the uniform-in-time convergence of \eqref{eq:system-n} to \eqref{eq:system} when $G_n \to G$ in the cut metric. 
A uniform-in-time LLN for marginal distributions is established in Theorem \ref{thm:moment-convergence-uniform}, which says that the empirical measure $\emppre^n(t)$ of $n$ particles $\{X_i^n : i=1,\dotsc,n\}$ converges to the averaged distribution of a continuum of particles $\{\Xlimit_u : u \in [0,1]\}$.
The proof relies on a truncation and approximation argument for the drift coefficients (Lemma \ref{lem:truncation}), and certain generalization (Lemma \ref{lem:Wasserstein-empirical}) of bounds on the Wasserstein distance between i.i.d.\ random variables and their common distribution established in \cite{FournierGuillin2015rate}.
In Theorem \ref{thm:moment-convergence-uniform-Lipschitz}, we strengthen Theorem \ref{thm:moment-convergence-uniform} with additional assumptions on graphons, to further obtain a uniform-in-particle convergence, LLN, and POC, all uniformly in time.
Theorem \ref{thm:moment-convergence-uniform} or \ref{thm:moment-convergence-uniform-Lipschitz}, together with the exponential ergodicity, guarantees the interchange of limits of large $n$ and $t$ for empirical measures and the convergence of stationary measures (Corollaries \ref{cor:interchange}, \ref{cor:interchange-Lipschitz} and \ref{cor:interchange-infinity}).

Our last main result is on the tractable computation of the graphon particle system and its stationary distributions. 
We study the Euler scheme associated with \eqref{eq:system-n}, and obtain a uniform-in-time bound (Theorem \ref{thm:moment-convergence-uniform-Euler}) for errors arising from the discretization.
Under certain conditions on the graphons, explicit rates of convergence are obtained in Corollaries \ref{cor:uniform-approximation-Euler} and \ref{cor:uniform-approximation-Euler-Lipschitz}, of the empirical measure of the Euler discretization to the stationary distribution of the graphon particle system \eqref{eq:system}.
This in particular answers the second question above.

\subsection{Organization}

The paper is organized as follows.
In Section \ref{sec:model} we state the space of graphons, the standing assumptions, and well-posedness of systems \eqref{eq:system} and \eqref{eq:system-n}. 
In Section \ref{sec:system-graphon} we study the long time behavior of the graphon particle system \eqref{eq:system}.
The exponential ergodicity and stationary distribution are shown in Theorem \ref{thm:exponential-ergodicity}.
In Section \ref{sec:system-n} we analyze the long time behavior of the finite particle system \eqref{eq:system-n}.
The quenched and annealed exponential ergodicity and quenched stationary distribution are shown in Theorem \ref{thm:exponential-ergodicity-n}.
In Section \ref{sec:LLN} we study the uniform-in-time convergence of the system \eqref{eq:system-n} to the system \eqref{eq:system}.
LLN and POC are given in Theorems \ref{thm:moment-convergence-uniform} and \ref{thm:moment-convergence-uniform-Lipschitz}.
We also show the interchange of limits and the convergence of limiting distributions in Corollaries \ref{cor:interchange}, \ref{cor:interchange-Lipschitz} and \ref{cor:interchange-infinity}.
In Section \ref{sec:Euler} we introduce the Euler discretization.
The convergence that is uniform in time and the number of particles is shown in Theorem \ref{thm:moment-convergence-uniform-Euler}.
The rate of convergence in given in Corollaries \ref{cor:uniform-approximation-Euler} and \ref{cor:uniform-approximation-Euler-Lipschitz}.
Finally Section \ref{sec:pf} collects the proofs of results in Sections \ref{sec:system-graphon}--\ref{sec:Euler}.

We close this section by introducing some frequently used notation.

\subsection{Notation}
\label{sec:notation}
Given a Polish space $\Smb$, denote by $\Bmc(\Smb)$ the Borel $\sigma$-field. 
Let $\Pmc(\Smb)$ be the space of probability measures on $\Smb$ endowed with the topology of weak convergence.
For a measurable function $f \colon \Smb \to \Rmb$, let $\|f\|_\infty := \sup_{x \in \Smb} |f(x)|$.
Denote by $\Cmb([0,\infty):\Smb)$ (resp.\ $\Cmb([0,T]:\Smb)$ for $T \in (0,\infty)$) the space of continuous functions from $[0,\infty)$ (resp.\ $[0,T]$) to $\Smb$, endowed with the topology of uniform convergence on compacts (resp.\ uniform convergence).
We will use $C$ to denote various positive constants in the paper and $C_m$ to emphasize the dependence on some parameter $m$.
Their values may change from line to line.
The probability law of a random variable $X$ will be denoted by $\Lmc(X)$.
Expectations under $\Pmb$ will be denoted by $\Emb$.
To simplify the notation, we will usually write $\Emb[X^k]$ as $\Emb X^k$.
For vectors $x,y \in \Rd$, denote by $|x|$ the Euclidean norm and $x \cdot y$ the inner product.
Let $\Nmb_0 := \Nmb \cup \{0\}$.

Denote by $W_p$, $p \in \Nmb$, the Wasserstein-$p$ distance (cf.\ \cite[Chapter 6]{Villani2008optimal}) on $\Pmc(\Rmb^k)$, $k \in \Nmb$:
\begin{align*}
	W_p(m_1,m_2) & := \left( \inf_\pi \int_{\Rmb^k \times \Rmb^k} |x-y|^p \, \pi(dx\,dy) \right)^{1/p}, \quad m_1,m_2 \in \Pmc(\Rmb^k),
\end{align*}
where the infimum is taken over all probability measures $\pi \in \Pmc(\Rmb^k \times \Rmb^k)$ with marginals $m_1$ and $m_2$, that is, $\pi(\cdot \times \Rmb^k) = m_1(\cdot)$ and $\pi(\Rmb^k \times \cdot) = m_2(\cdot)$.
It is well-known that (cf.\ \cite[Remarks 6.5 and 6.6]{Villani2008optimal})
\begin{align}
	W_p(m_1,m_2) & \ge W_1(m_1,m_2) \notag \\
	& = \sup \left\{ \int_{\Rmb^k} \phi(x)\,m_1(dx) - \int_{\Rmb^k} \phi(x)\,m_2(dx) \,\Big|\, \phi \colon \Rmb^k \to \Rmb \mbox{ is } 1\mbox{-Lipschitz} \right\}. \label{eq:Wasserstein-duality}
\end{align}
Note that
\begin{equation}
	\label{eq:Wasserstein-joint}
	W_2^2(m, \mtil) \ge W_2^2(m_1, \mtil_1) + W_2^2(m_2, \mtil_2)
\end{equation}
for any $m, \mtil \in \Pmc(\Rmb^{k_1} \times \Rmb^{k_2})$ with marginals $m_1,\mtil_1 \in \Pmc(\Rmb^{k_1})$ and $m_2,\mtil_2 \in \Pmc(\Rmb^{k_2})$ respectively, where $k_1,k_2 \in \Nmb$.
In addition, if $m=m_1 \otimes m_2$ and $\mtil=\mtil_1 \otimes \mtil_2$ are product measures, then
\begin{equation}
	\label{eq:Wasserstein-product}
	W_2^2(m, \mtil) = W_2^2(m_1, \mtil_1) + W_2^2(m_2, \mtil_2).
\end{equation}

\section{Model and assumptions}
\label{sec:model}

We follow the notation used in \cite[Chapters 7 and 8]{Lovasz2012large}.
Let $I := [0,1]$.
Denote by $\Gmc$ the space of all bounded symmetric measurable functions $G \colon I \times I \to \Rmb$.
A graphon $G$ is an element of $\Gmc$ with $0 \le G \le 1$.
The cut norm on $\Gmc$ is defined by
$$\|G\|_\square := \sup_{S,T \in \Bmc(I)} \left| \int_{S \times T} G(u,v)\,du\,dv \right|,$$
and the corresponding cut metric and cut distance are defined by 
$$d_\square(G_1,G_2) := \|G_1-G_2\|_\square, \quad \delta_\square(G_1,G_2) := \inf_{\varphi \in S_I} \|G_1-G_2^\varphi\|_\square,$$
where $S_I$ denotes the set of all invertible measure preserving maps $I \to I$, and $G^\varphi(u,v):=G(\varphi(u),\varphi(v))$.

\begin{Remark}
	\label{rmk:graphon-convergence}
%	\begin{enumerate}[(a)]
%	\item
	We will also view a graphon $G$ as an operator from $L^\infty(I)$ to $L^1(I)$ with the operator norm
	$$\|G\|:=\|G\|_{\infty\to1}:=\sup_{\|g\|_\infty \le 1} \|Gg\|_1 = \sup_{\|g\|_\infty \le 1} \int_I \left| \int_I G(u,v) g(v) \, dv \right| du.$$	
	From \cite[Lemma 8.11]{Lovasz2012large} it follows that if $\|G_n-G\|_\square \to 0$ for a sequence of graphons $G_n$, then $\|G_n-G\| \to 0$.
%	\end{enumerate}	
\end{Remark}

Given a graphon $G \in \Gmc$ and a collection of initial distributions $\lawlimit(0) := (\lawlimit_u(0) \in \Pmc(\Rd) : u \in I)$,
recall the graphon particle system \eqref{eq:system} and the finite particle system \eqref{eq:system-n}.
The following assumptions will be made throughout the paper.

\vspace{.08in}
\noindent\textbf{Standing Assumptions:}
%\begin{Condition}
%	\phantomsection
%	\label{cond:basic}
	\begin{itemize}
	\item 
		The map $I \ni u \mapsto \lawlimit_u(0):=\Lmc(\Xlimit_u(0)) \in \Pmc(\Rd)$ is measurable, and $\sup_{u \in I} \Emb |\Xlimit_u(0)|^4 < \infty$.
	\item 
		The drift functions $f$ and $b$ are Lipschiz with Lipschitz constant $K_f$ and $K_b$, respectively, namely
		\begin{align*}
			|f(x_1)-f(x_2)| & \le K_f |x_1-x_2|, \quad \forall\, x_1,x_2 \in \Rd, \\
			|b(x_1,y_1)-b(x_2,y_2)| & \le K_b (|x_1-x_2| +|y_1-y_2|), \quad \forall\, x_1,x_2,y_1,y_2 \in \Rd.
		\end{align*}
	\item 
		Dissipativity:
		There exists some $c_0 \in (0,\infty)$ such that
		\begin{equation}
			\label{eq:c0}
			(x_1-x_2) \cdot (f(x_1)-f(x_2)) \le -c_0 |x_1-x_2|^2, \quad \forall\, x_1,x_2 \in \Rd
		\end{equation}
		and
		\begin{equation}
			\label{eq:kappa}
			\kappa := c_0-2K_b > 0.
		\end{equation}
	\item
		$G_n \in \Gmc$ is a graphon and
		\begin{enumerate}[(i)]
		\item either $\xi_{ij}^n=G_n(\frac{i}{n},\frac{j}{n})$,
		\item or $\xi_{ij}^n=\xi_{ji}^n=\text{Bernoulli}(G_n(\frac{i}{n},\frac{j}{n}))$ independently for $1 \le i \le j \le n$, and independent of $\{\Xlimit_u(0),B_u : u \in I\}$.
		\end{enumerate}
	\end{itemize} 	
%\end{Condition}

\begin{Remark}
	\phantomsection
	\label{rmk:linear-growth}
	\begin{enumerate}[(a)]
	\item
		The finite forth moment on $\Xlimit(0)$ is assumed to obtain Wasserstein-$2$ estimates in Sections \ref{sec:LLN} and \ref{sec:Euler}. A weaker condition such as a finite second moment on $\Xlimit(0)$ would be sufficient to establish exponential ergodicity properties in Sections \ref{sec:system-graphon} and \ref{sec:system-n}.
	\item 
		Clearly, $f$ and $b$ have linear growth, namely there exists some $C \in (0,\infty)$ such that $|f(x)| + |b(x,y)| \le C(1+|x|+|y|)$ for all $x,y \in \Rd$.
	\item
		A common example of $b$ and $f$ satisfying \eqref{eq:c0} and \eqref{eq:kappa} is linear (as in the study of linear quadratic graphon mean-field games in, e.g., \cite{GaoTchuendomCaines2020linear}) and mean-reverting:
		$$f(x) + b(x,y)=-c_1x+c_2y, \quad \text{for some } c_1 > c_2 > 0.$$
		In particular, the choice of $f(x)=-(c_1+c_2)x$ and $b(x,y)=c_2(x+y)$ satisfies \eqref{eq:c0} and \eqref{eq:kappa} since $c_0=c_1+c_2 > 2c_2=2K_b$.
	\end{enumerate}	
\end{Remark}

The following result gives well-posedness of systems \eqref{eq:system} and \eqref{eq:system-n}.

\begin{Proposition}
	\phantomsection
	\label{prop:well-posedness}
	\begin{enumerate}[(a)] 
	\item
		There exists a unique pathwise solution to \eqref{eq:system}.
		For every $T < \infty$, the map $I \ni u \mapsto \lawlimit_u \in \Pmc(\Cmb([0,T]:\Rd))$ is measurable and 
		\begin{equation*}
			\sup_{u \in I} \sup_{t \in [0,T]} \Emb \left[ |\Xlimit_u(t)|^4 \right] < \infty.
		\end{equation*}
	\item 
		There exists a unique pathwise solution to \eqref{eq:system-n}.
		Also for every $T < \infty$, 
		\begin{equation*}
			\max_{i=1,\dotsc,n} \sup_{t \in [0,T]} \Emb \left[ |X_i^n(t)|^4 \right] < \infty.
		\end{equation*}
	\end{enumerate}	
\end{Proposition}

The proof of Proposition \ref{prop:well-posedness} is standard (see e.g.\ \cite{Sznitman1991} and \cite{BayraktarChakrabortyWu2020graphon} for part (a), and \cite[Theorems 5.2.5 and 5.2.9]{KaratzasShreve1991brownian} for part (b)) and hence is omitted.

\section{Exponential ergodicity of the graphon particle system}
\label{sec:system-graphon}

In this section we show the exponential ergodicity for the graphon particle system \eqref{eq:system}.

First recall the standing assumptions in Section \ref{sec:model}.
The following result guarantees that Proposition \ref{prop:well-posedness}(a) holds uniformly in time.

\begin{Proposition}
	\label{prop:finite-moment-uniform}
	The system \eqref{eq:system} has finite fourth moments uniformly in time, namely
	\begin{equation*}
		\sup_{u \in I} \sup_{t \ge 0} \Emb \left[ |\Xlimit_u(t)|^4 \right] < \infty.
	\end{equation*}
\end{Proposition}

Next we introduce some notations before stating the exponential ergodicity property.
For $\eta := (\eta_u : u \in I) \in [\Pmc(\Rmb^d)]^I$ with $\sup_{u \in I} \int_{\Rmb^d} |x|^4\,\eta_u(dx) < \infty$, consider the system $\Ybar^{\eta} = (\Ybar_u^{\eta} : u \in I)$ given by
\begin{align}
	\Ybar_u^{\eta}(t) & = \Ybar_u^{\eta}(0) + \int_0^t \left( f(\Ybar_u^{\eta}(s)) + \int_I \int_{\Rmb^d} b(\Ybar_u^{\eta}(s),x)G(u,v) \,\lawlimit_{v,s}^{\eta}(dx)\,dv\right)ds \label{eq:Ybar_eta} \\
	& \quad + \sigma B_u(t), \quad \lawlimit_{u,t}^{\eta}=\Lmc(\Ybar_u^{\eta}(t)), \quad u \in I, \notag
\end{align}	
where $(\Ybar^{\eta}_u(0) : u \in I)$ are mutually independent and also independent of $\{B_u : u \in I\}$ with $\Lmc(\Ybar^{\eta}_u(0))=\eta_u$.
Note that $\Ybar^{\eta}$ is well-defined and $\sup_{u \in I} \sup_{t \ge 0} \int_{\Rmb^d} |x|^4\,\lawlimit_{u,t}^{\eta}(dx) < \infty$ by Proposition \ref{prop:well-posedness}(a).
Denote by $P_t$ the associated Markov semigroup:
\begin{equation}
	\label{eq:Pt}
	P_t \eta := \Lmc(\Ybar^{\eta}(t)), \quad t \ge 0.
\end{equation}

The following theorem shows that $\lawlimit_{u,t}$ (and its average) converges exponentially fast to the limiting distribution, which is also invariant with respect to $P_t$.

\begin{Theorem}
	\phantomsection
	\label{thm:exponential-ergodicity}
	\begin{enumerate}[(a)]
	\item 
		There exists a unique collection of probability measures $(\lawlimit_{u,\infty} : u \in I)$ such that
		\begin{equation}
			\label{eq:exponential-ergodicity}
			\sup_{u \in I} W_2(\lawlimit_{u,t},\lawlimit_{u,\infty}) \le \sqrt{4\kappa_1\frac{c_0-K_b}{\kappa}} e^{-\kappa t/2}, \quad t \ge 0,
		\end{equation}
		and hence
		\begin{equation}
			\label{eq:exponential-ergodicity-average}
			W_2(\lawlimitmean(t),\lawlimitmean(\infty)) \le \sqrt{4\kappa_1\frac{c_0-K_b}{\kappa}} e^{-\kappa t/2}, \quad t \ge 0,
		\end{equation}
		where $\kappa_1 := \sup_{u \in I} \sup_{t \ge 0} \Emb |\Xlimit_u(t)|^2$ and the averaged measures $\lawlimitmean(t)$ and $\lawlimitmean(\infty)$ are defined as
		\begin{equation}
			\label{eq:lawlimitmean}
			\lawlimitmean(t) := \int_I \lawlimit_{u,t} \, du, \quad \lawlimitmean(\infty) := \int_I \lawlimit_{u,\infty} \, du.
		\end{equation}
	\item
		The collection $\lawlimit(\infty):=(\lawlimit_{u,\infty} : u \in I)$ is invariant with respect to the Markov semigroup $P_t$ defined in \eqref{eq:Pt}, namely
		\begin{equation*}
%			\label{eq:invariant}
			P_t\lawlimit(\infty) = \lawlimit(\infty), \quad t \ge 0.
		\end{equation*}
	\item
		There exists some $C \in (0,\infty)$ such that
		\begin{equation*}
%			\label{eq:continuity}
			\sup_{t \in [0,\infty]} W_2(\lawlimit_{u_1,t},\lawlimit_{u_2,t}) \le \max \left\{ W_2(\lawlimit_{u_1,0},\lawlimit_{u_2,0}), C \int_I |G(u_1,v)-G(u_2,v)| \,dv \right\}, \quad u_1,u_2 \in I.
		\end{equation*}
	\end{enumerate}	
\end{Theorem}

Proofs of Proposition \ref{prop:finite-moment-uniform} and Theorem \ref{thm:exponential-ergodicity} are given in Section \ref{sec:pf-system-graphon}.

An immediate consequence of Theorem \ref{thm:exponential-ergodicity}(c) is that the marginal distribution is (Lipschitz) continuous as long as the initial distribution and the graphon are so.

\begin{Condition}
	\label{cond:G-continuous}
	There exists a finite collection of intervals $\{I_i : i=1,\dotsc,N\}$ for some $N \in \Nmb$, such that $\cup_{i=1}^N I_i = I$ and for each $i \in \{1,\dotsc,N\}$:
	\begin{enumerate}[(a)]
	\item The map $I_i \ni u \mapsto \lawlimit_u(0) \in \Pmc(\Rd)$ is continuous with respect to the $W_2$ metric.
	\item For each $u \in I_i$, there exists a subset $A_u \subset I$ such that $\lambda_I(A_u)=0$ and $G(u,v)$ is continuous at $(u,v) \in I \times I$ for each $v \in I \setminus A_u$, where $\lambda_I$ denotes the Lebesgue measure on $I$.
	\end{enumerate}
\end{Condition} 

\begin{Condition}
	\label{cond:G-Lipschitz}
	There exist some $K_G \in (0,\infty)$ and a finite collection of intervals $\{I_i : i=1,\dotsc,N\}$ for some $N \in \Nmb$, such that $\cup_{i=1}^N I_i = I$ and
	\begin{align*}
		W_2(\mu_{u_1}(0),\mu_{u_2}(0)) & \le K_G |u_1-u_2|, \quad u_1,u_2 \in I_i, \quad i \in \{1,\dotsc,N\}, \\
		|G(u_1,v_1)-G(u_2,v_2)| & \le K_G (|u_1-u_2|+|v_1-v_2|), \: (u_1,v_1),(u_2,v_2) \in I_i \times I_j, \: i,j \in \{1,\dotsc,N\}.
	\end{align*}
\end{Condition}

\begin{Corollary}
	\phantomsection
	\label{cor:continuity-Lipschitz-special}
	\begin{enumerate}[(a)]
	\item 
		Suppose Condition \ref{cond:G-continuous} holds. 
		Then for each $i \in \{1,\dotsc,N\}$, $\sup_{t \in [0,\infty]} W_2(\lawlimit_{u_1,t},\lawlimit_{u_2,t}) \to 0$ whenever $u_1 \to u_2$ in $I_i$.
	\item
		Suppose Condition \ref{cond:G-Lipschitz} holds. 
		Then there exists some $C \in (0,\infty)$ such that 
		$$\sup_{t \in [0,\infty]} W_2(\lawlimit_{u_1,t},\lawlimit_{u_2,t}) \le C|u_1-u_2|$$ 
		whenever $u_1,u_2 \in I_i$ for some $i \in \{1,\dotsc,N\}$.
	\end{enumerate}	
\end{Corollary}

\begin{proof}
%[Proof of Corollary \ref{cor:continuity-Lipschitz-special}]
	This is immediate from Theorem \ref{thm:exponential-ergodicity}(c).	
\end{proof}

We note that, as illustrated through the following example of Gaussian processes with linear coefficients, the graph structure plays a crucial role in the long-time behavior of the system, and hence the stationary measure, due to the heterogeneity of the system, is not necessarily tractable.
This is indeed one of the main reasons we are interested in the second question in the introduction, which is answered via the uniform-in-time convergence and Euler discretization in the next few sections.

\begin{Example}
	\label{eg:Gaussian}
	Suppose $d=1$.
	Suppose $f$ and $b$ are linear, namely $f(x)=c_1-c_2x$ and $b(x,y)=c_3+c_4x+c_5y$.
	Suppose $c_2>0$ and $c_2-2\max\{|c_4|,|c_5|\}>0$ so that the dissipativity assumption holds.
	Then \eqref{eq:system} is a collection of Gaussian processes.
	Letting $m_u(t):=\Emb[\Xlimit_u(t)]$ and $M_u(t):=\Emb[\Xlimit_u^2(t)]$, we have
	\begin{align*}
		m_u(t) & = m_u(0) + \int_0^t \left( c_1-c_2m_u(s) + \int_0^1 (c_3+c_4m_u(s)+c_5m_v(s))G(u,v) \,dv\right)ds, \\
		M_u(t) & = M_u(0) + \Emb\left[\int_0^t 2\Xlimit_u(s)\,d\Xlimit_u(s)\right] + \sigma^2t \\
		& = M_u(0) + 2\int_0^t \left( c_1m_u(s)-c_2M_u(s) \right. \\
		& \quad \left. + \int_0^1 (c_3m_u(s)+c_4M_u(s)+c_5m_u(s)m_v(s))G(u,v)\,dv\right)ds + \sigma^2t. 
	\end{align*}
	By Theorem \ref{thm:exponential-ergodicity}, the limits $m_u(\infty):= \lim_{t \to \infty} m_u(t)$ and $M_u(\infty):= \lim_{t \to \infty} M_u(t)$ exist and should satisfy the following equations:
	\begin{align*}
		c_1-c_2m_u(\infty) + \int_0^1 (c_3+c_4m_u(\infty)+c_5m_v(\infty))G(u,v) \,dv & = 0, \\
		c_1m_u(\infty)-c_2M_u(\infty) + \int_0^1 (c_3m_u(\infty)+c_4M_u(\infty)+c_5m_u(\infty)m_v(\infty))G(u,v)\,dv + \frac{1}{2}\sigma^2 & = 0. 
	\end{align*}
	Here the first equation is a Fredholm integral equation of the second kind, from which $m_u(\infty)$ could be written as a Liouville–Neumann series, and $M_u(\infty)$ could then be solved from the second equation.
	
	We note that even in this setup of linear systems, the long time behavior crucially depends on the graphon $G$ and the stationary distribution, such as the mean $m_u(\infty)$, is not necessarily explicit or tractable.
	In some special cases, one can get explicit expressions. For example, if we further assume $c_1=c_3=0$, then we can get
	\begin{equation*}
		m_u(\infty)=0, \quad M_u(\infty)=\frac{\sigma^2}{2(c_2-c_4\int_0^1 G(u,v)\,dv)},
	\end{equation*}
	and the second moment of the averaged measure $\lawlimitmean(\infty)$ is
	\begin{equation*}
		\int_\Rmb x^2 \,\lawlimitmean(\infty)(dx) = \int_0^1 M_u(\infty)\,du = \frac{\sigma^2}{2} \int_0^1 \frac{1}{c_2-c_4\int_0^1 G(u,v)\,dv} \,du. 
	\end{equation*}
\end{Example}

\section{Exponential ergodicity of the finite particle system}
\label{sec:system-n}

In this section we establish the exponential ergodicity of the joint distribution for the finite particle system \eqref{eq:system-n}.

Using the standing assumptions in Section \ref{sec:model}, we first show that Proposition \ref{prop:well-posedness}(b) holds uniformly in time, in the quenched sense by conditioning on the random interactions $\xi_{ij}^n$, and hence also in the annealed sense.
Write $\xi^n := (\xi_{ij}^n)_{i,j=1}^n$ and 
\begin{equation*}
	\Emb^{n,z}[\,\cdot\,] := \Emb[\,\cdot\,|\, \xi^n=(z_{ij})_{i,j=1}^n], \quad (z_{ij}=z_{ji})_{i,j=1}^n \in [0,1]^{n^2}.
\end{equation*}

\begin{Proposition}
	\phantomsection
	\label{prop:finite-moment-uniform-n}
	There exists some constant $\kappa_2 \in (0,\infty)$ such that
	\begin{equation*}
		\sup_{n \in \Nmb} \max_{i=1,\dotsc,n} \sup_{t \ge 0} \Emb^{n,\xi^n} \left[ |X_i^n(t)|^2 \right] \le \kappa_2 \:\: a.s., \quad \sup_{n \in \Nmb} \max_{i=1,\dotsc,n} \sup_{t \ge 0} \Emb \left[ |X_i^n(t)|^2 \right] \le \kappa_2.
	\end{equation*}
\end{Proposition}

Next we introduce some notations before stating the exponential ergodicity property.
Define the annealed and quenched joint distributions at time $t \ge 0$ by
\begin{equation*}
	\jointpre^n(t) := \Lmc((X_i^n(t))_{i=1}^n) \in \Pmc((\Rmb^d)^n)
\end{equation*}
and
\begin{equation*}
	\jointpre^{n,z}(t) := \Lmc( (X_i^n(t))_{i=1}^n \,|\, \xi^n=(z_{ij})_{i,j=1}^n), \quad z=(z_{ij}=z_{ji})_{i,j=1}^n \in [0,1]^{n^2}.
\end{equation*}
For $\eta \in \Pmc((\Rmb^d)^n)$ and $z=(z_{ij}=z_{ji})_{i,j=1}^n \in [0,1]^{n^2}$, consider the system $Y^{n,z,\eta} = (Y_i^{n,z,\eta})_{i=1}^n$ given by
\begin{align}
	Y_i^{n,z,\eta}(t) & = Y_i^{n,z,\eta}(0) + \int_0^t \left( f(Y_i^{n,z,\eta}(s)) + \frac{1}{n} \sum_{j=1}^n z_{ij} b(Y_i^{n,z,\eta}(s),Y_j^{n,z,\eta}(s)) \right) ds \label{eq:Y_eta-n} \\
	& \quad + \sigma B_{\frac{i}{n}}(t), \quad i \in \{1,\dotsc,n\}, \notag
\end{align}	
where $Y^{n,z,\eta}(0)$ is independent of $\{B_u : u \in I\}$ with $\Lmc(Y^{n,z,\eta}(0))=\eta$.
Denote by $P_t^{n,z}$ the associated Markov semigroup:
\begin{equation}
	\label{eq:Ptn}
	P_t^{n,z} \eta := \Lmc(Y^{n,z,\eta}(t)), \quad t \ge 0, \quad \eta \in \Pmc((\Rmb^d)^n).
\end{equation}

The following theorem shows that $\jointpre^{n,z}(t)$ (resp.\ $\jointpre^n(t)$) converges exponentially fast to the limiting distribution, which is also invariant with respect to $P_t^{n,z}$.

\begin{Theorem}
	\phantomsection
	\label{thm:exponential-ergodicity-n}
	\begin{enumerate}[(a)]
	\item 
		There exists a unique collection of probability measures $\{\jointpre^{n,z}(\infty) : z=(z_{ij}=z_{ji})_{i,j=1}^n \in [0,1]^{n^2}\}$ such that
		\begin{equation}
			\label{eq:exponential-ergodicity-n-z}
			\sup_{n \in \Nmb} \frac{1}{\sqrt{n}} W_2(\jointpre^{n,\xi^n}(t),\jointpre^{n,\xi^n}(\infty)) \le \sqrt{4\kappa_2} e^{-\kappa t}, \quad t \ge 0, \quad a.s.,
		\end{equation}
		and hence
		\begin{equation}
			\label{eq:exponential-ergodicity-n}
			\sup_{n \in \Nmb} \frac{1}{\sqrt{n}} W_2(\jointpre^n(t),\jointpre^n(\infty)) \le \sqrt{4\kappa_2} e^{-\kappa t}, \quad t \ge 0,
		\end{equation}	
		where $\kappa_2$ is as in Proposition \ref{prop:finite-moment-uniform-n}, and
		\begin{equation}
			\label{eq:theta-infinity}
			\jointpre^n(\infty) := \Emb[\jointpre^{n,\xi^n}(\infty)].
		\end{equation}	
	\item
		The joint distribution $\jointpre^{n,z}(\infty)$ is invariant with respect to the Markov semigroup $P_t^{n,z}$ defined in \eqref{eq:Ptn}, namely
		\begin{equation*}
%			\label{eq:invariant-n}
			P_t^{n,z}\jointpre^{n,z}(\infty) = \jointpre^{n,z}(\infty), \quad t \ge 0.
		\end{equation*}
	\end{enumerate}
\end{Theorem}

Proofs of Proposition \ref{prop:finite-moment-uniform-n} and Theorem \ref{thm:exponential-ergodicity-n} are given in Section \ref{sec:pf-system-n}.

\section{Uniform-in-time convergence}
\label{sec:LLN}

In this section we analyze the uniform-in-time convergence of the finite particle system \eqref{eq:system-n} to the graphon particle system \eqref{eq:system}.

We make the following assumption on the kernel $G_n$. 
Note that \eqref{eq:G_n-step-graphon} is just a convenient and natural form to view $(G_n(\frac{i}{n},\frac{j}{n}) : i,j=1,\dotsc,n)$ as a piece-wise constant graphon.

\begin{Condition}
	\label{cond:G_n-step-graphon}
	$G_n$ is a step graphon, that is,
	\begin{equation}
		\label{eq:G_n-step-graphon}
		G_n(u,v) = G_n \left( \frac{\lceil nu \rceil}{n}, \frac{\lceil nv \rceil}{n} \right), \quad \text{for } (u,v) \in I \times I.
	\end{equation}
	Moreover, $G_n \to G$ in the cut metric as $n \to \infty$.
\end{Condition}

\begin{Remark}
	In general, if $\delta_\square(G_n,G) \to 0$ for a sequence of step graphons, then it follows from \cite[Theorem 11.59]{Lovasz2012large} that $\|G_n-G\|_\square \to 0$, after suitable relabeling of $G_n$.
	Therefore we directly assume in Condition \ref{cond:G_n-step-graphon} that the convergence of $G_n$ to $G$ is in the cut metric $d_{\square}$, instead of assuming that $d_{\square}(G_n^{\varphi_n},G) \to 0$ for some relabeling function $\varphi_n$ and $X_i^n(0)=\Xlimit_{\varphi_n^{-1}(\frac{i}{n})}(0)$.	
\end{Remark}

The following convergence on finite time intervals was shown in \cite[Theorem 3.1]{BayraktarChakrabortyWu2020graphon} without the dissipativity assumption.

\begin{Proposition}(\cite[Theorem 3.1]{BayraktarChakrabortyWu2020graphon})
	\phantomsection
	\label{thm:old}
	Suppose Conditions \ref{cond:G-continuous} and \ref{cond:G_n-step-graphon} hold.  
	Fix $T \in (0,\infty)$.
	As $n \to \infty$,
	\begin{equation*}
%			\label{eq:moment-convergence-uniform}
		\frac{1}{n} \sum_{i=1}^n \Emb \left[\sup_{t \in [0,T]} \left|X_i^n(t)-\Xlimit_{\frac{i}{n}}(t)\right|^2\right] \to 0.
	\end{equation*}
\end{Proposition}

Recall $\lawlimitmean(t)$ and $\lawlimitmean(\infty)$ introduced in \eqref{eq:lawlimitmean}.
Let
\begin{equation*}
	\emppre^n(t) := \frac{1}{n} \sum_{i=1}^n \delta_{X_i^n(t)}, \quad \lawpre^n(t) := \frac{1}{n} \sum_{i=1}^n \Lmc(X_i^n(t)) = \Emb \emppre^n(t).
\end{equation*}
Although Proposition \ref{thm:old} holds for finite time horizon, it does not provide sufficient information about the convergence of stationary measures.
Under the dissipativity assumption, we have the following uniform in time convergence of $X_i^n$ and LLN of $\emppre^n$ and $\lawpre^n$, which in particular guarantees the convergence of stationary measures (see Corollary \ref{cor:interchange-infinity}).

\begin{Theorem}
	\phantomsection
	\label{thm:moment-convergence-uniform}
	Suppose Conditions \ref{cond:G-continuous} and \ref{cond:G_n-step-graphon} hold. 
	\begin{enumerate}[(a)]
	\item 
		As $n \to \infty$,
		\begin{equation*}
%			\label{eq:moment-convergence-uniform}
			\sup_{t \ge 0} \frac{1}{n} \sum_{i=1}^n \Emb |X_i^n(t)-\Xlimit_{\frac{i}{n}}(t)|^2 \to 0.
		\end{equation*}	
	\item
		(LLN) As $n \to \infty$,	
		\begin{equation}
			\label{eq:LLN-uniform}
			\sup_{t \ge 0} W_2(\lawpre^n(t),\lawlimitmean(t)) \to 0, \quad \sup_{t \ge 0} \Emb W_2(\emppre^n(t),\lawlimitmean(t)) \to 0.
		\end{equation}
	\end{enumerate}
\end{Theorem}

\begin{Remark}
	\begin{enumerate}[(a)]
	\item 
		We note that Theorem \ref{thm:moment-convergence-uniform} and many existing results (such as \cite[Theorem 2.1(b)]{BarreDobsonOttobreZatorska2020fast}, \cite[Theorem 3.4]{BudhirajaFan2017}, \cite[Theorem 3.4]{BudhirajaPalMajumder2015long} and \cite[Theorem 2]{Veretennikov2006ergodic}) work on the marginal distributions, which is sufficient for the analysis of approximating stationary measure in Corollary \ref{cor:uniform-approximation-Euler-Lipschitz}.
		The study on the trajectory level of the difference $\Emb \left[ \sup_{t \ge 0} \frac{1}{n} \sum_{i=1}^n |X_i^n(t)-\Xlimit_{\frac{i}{n}}(t)|^2 \right]$ is more challenging and beyond the scope of this work.
	\item
		We also note that graphs with vanishing density degrees and rescaled strength of interactions are analyzed in \cite{BayraktarChakrabortyWu2020graphon} and Proposition \ref{thm:old} is proved under certain conditions via a Girsanov's change of measure argument.
		It is challenging to apply such an argument to the long-time analysis. 
		The study of the uniform-in-time convergence (and Euler discretization) for such graphs will be the future work.
	\end{enumerate}	
\end{Remark}

The following condition will be used for analyzing efficient Euler discretization and simulation in Section \ref{sec:Euler}.

\begin{Condition}
	\label{cond:G_n=G}
	$G_n$ is a graphon such that $G_n(\frac{i}{n},\frac{j}{n}) = G(\frac{i}{n},\frac{j}{n})$ for each $i,j \in \{1,\dotsc,n\}$.
\end{Condition}

\begin{Remark}
	We note that Condition \ref{cond:G_n=G} is trivially satisfied if $G_n=G$.
	Alternatively, one may take $G_n$ to be a step graphon that is consistent with $G$:
	\begin{equation*}
		G_n(u,v) = G \left( \frac{\lceil nu \rceil}{n}, \frac{\lceil nv \rceil}{n} \right), \quad \text{for } (u,v) \in I \times I.
	\end{equation*}
\end{Remark}

Condition \ref{cond:G_n=G} allows one to obtain POC and the rate of convergence in Theorem \ref{thm:moment-convergence-uniform}.
Let
\begin{equation}
	\label{eq:an}
	a(n) := n^{-1/d}+n^{-1/12}.
\end{equation}

\begin{Theorem}
	\phantomsection
	\label{thm:moment-convergence-uniform-Lipschitz}
	Suppose Conditions \ref{cond:G-Lipschitz} and \ref{cond:G_n=G} hold. 
	Then there exists some $C \in (0,\infty)$ such that the following hold.
	\begin{enumerate}[(a)]
	\item 
		For all $n \in \Nmb$,
		\begin{equation*}
%			\label{eq:moment-convergence-uniform-Lipschitz}
			\sup_{t \ge 0} \max_{i=1,\dotsc,n} \Emb |X_i^n(t)-\Xlimit_{\frac{i}{n}}(t)|^2 \le \frac{C}{n}.
		\end{equation*}
	\item
		(LLN) For all $n \in \Nmb$,
		\begin{equation*}
%			\label{eq:LLN-uniform-Lipschitz}
			\sup_{t \ge 0} W_2(\lawpre^n(t),\lawlimitmean(t)) \le \frac{C}{\sqrt{n}}, \quad \sup_{t \ge 0} \Emb W_2(\emppre^n(t),\lawlimitmean(t)) \le Ca(n).
		\end{equation*}
	\item
		(POC) For all $n,k \in \Nmb$ and any distinct $i_1,\dotsc,i_k \in \{1,\dotsc,n\}$,
		\begin{equation*}
%			\label{eq:POC-uniform}
			\sup_{t \ge 0} W_2(\Lmc(X_{i_1}^n(t), \dotsc, X_{i_k}^n(t)), \, \lawlimit_{\frac{i_1}{n},t} \otimes \dotsb \otimes \lawlimit_{\frac{i_k}{n},t}) \le \frac{C\sqrt{k}}{\sqrt{n}}.
		\end{equation*}
	\end{enumerate}
\end{Theorem}

Proofs of Theorems \ref{thm:moment-convergence-uniform} and \ref{thm:moment-convergence-uniform-Lipschitz} are given in Section \ref{sec:pf-LLN}.

\begin{Remark}
	\label{rmk:an}
	The rate $a(n)$ is related to the upper bound of the Wasserstein distance between the empirical measure of independent random variables and their averaged distribution.
	It may be replaced by other function of $n$ that vanishes faster, as a result of which the constant $C$ in Theorem \ref{thm:moment-convergence-uniform-Lipschitz}(b) will be larger; see Remark \ref{rmk:choice-of-p}.
\end{Remark}

As an immediate consequence of the exponential ergodicity of the graphon particle system \eqref{eq:system} and the uniform-in-time convergence, one has the interchange of limits as $t \to \infty$ and $n \to \infty$.

\begin{Corollary}
%	\phantomsection
	\label{cor:interchange}
%	\begin{enumerate}[(a)]
%	\item 
	Suppose Conditions \ref{cond:G-continuous} and \ref{cond:G_n-step-graphon} hold.
	Then
%	\end{enumerate}
	\begin{equation*}
		\lim_{n,t \to \infty} W_2(\lawpre^n(t),\lawlimitmean(\infty)) = 0, \quad \lim_{n,t \to \infty} \Emb W_2(\emppre^n(t),\lawlimitmean(\infty)) = 0.
	\end{equation*}	
\end{Corollary}

\begin{proof}
	This follows from \eqref{eq:exponential-ergodicity-average} and Theorem \ref{thm:moment-convergence-uniform}(b).
\end{proof}

\begin{Corollary}
	\phantomsection
	\label{cor:interchange-Lipschitz}
	Suppose Conditions \ref{cond:G-Lipschitz} and \ref{cond:G_n=G} hold.
	Then there exists $C \in (0,\infty)$ such that the following hold.
	\begin{enumerate}[(a)]
	\item 
		For all $n \in \Nmb$ and $t \ge 0$,
		\begin{equation*}
	%		\label{eq:interchange-Lipschitz}
			W_2(\lawpre^n(t),\lawlimitmean(\infty)) \le C\left(\frac{1}{\sqrt{n}} + e^{-\kappa t/2}\right), \quad
			\Emb W_2(\emppre^n(t),\lawlimitmean(\infty)) \le C\left(a(n) + e^{-\kappa t/2}\right).
		\end{equation*}
		In particular,
		\begin{equation*}
			\lim_{n,t \to \infty} W_2(\lawpre^n(t),\lawlimitmean(\infty)) = 0, \quad \lim_{n,t \to \infty} \Emb W_2(\emppre^n(t),\lawlimitmean(\infty)) = 0.
		\end{equation*}	
	\item
		For all $n,k \in \Nmb$, $t \ge 0$ and any distinct $i_1,\dotsc,i_k \in \{1,\dotsc,n\}$,
		\begin{equation*}
%			\label{eq:interchange-Lipschitz-3}
			W_2(\Lmc(X_{i_1}^n(t), \dotsc, X_{i_k}^n(t)), \, \lawlimit_{\frac{i_1}{n},\infty} \otimes \dotsb \otimes \lawlimit_{\frac{i_k}{n},\infty}) \le C\sqrt{k} \left( \frac{1}{\sqrt{n}} + e^{-\kappa t/2} \right).
		\end{equation*}	
	\end{enumerate}
\end{Corollary}

\begin{proof}
	(a) This follows from \eqref{eq:exponential-ergodicity-average} and Theorem \ref{thm:moment-convergence-uniform-Lipschitz}(b).		
	
	(b) This follows from Theorem \ref{thm:moment-convergence-uniform-Lipschitz}(c), \eqref{eq:Wasserstein-product} and \eqref{eq:exponential-ergodicity}.	
\end{proof}

From Theorem \ref{thm:exponential-ergodicity-n}(a) we know that the limiting distribution
\begin{equation*}
	\lawpre^n(\infty) := \lim_{t \to \infty} \lawpre^n(t) = \lim_{t \to \infty} \frac{1}{n} \sum_{i=1}^n \Lmc(X_i^n(t))
\end{equation*}
is well-defined.
The following corollary shows that $\lawpre^n(\infty)$ converges to the averaged long time distribution $\lawlimitmean(\infty)$ of the graphon particle system \eqref{eq:system}.

\begin{Corollary}
	\phantomsection
	\label{cor:interchange-infinity}
	\begin{enumerate}[(a)]
	\item 
		Suppose Conditions \ref{cond:G-continuous} and \ref{cond:G_n-step-graphon} hold. 
		Then
		\begin{equation*}
			\lim_{n \to \infty} W_2(\lawpre^n(\infty), \lawlimitmean(\infty)) = 0.
		\end{equation*}
	\item
		Suppose Conditions \ref{cond:G-Lipschitz} and \ref{cond:G_n=G} hold.
		Then there exists $C \in (0,\infty)$ such that
		\begin{equation*}
			W_2(\lawpre^n(\infty), \lawlimitmean(\infty)) \le \frac{C}{\sqrt{n}}.
		\end{equation*}	
		In addition, for all $n,k \in \Nmb$, $t \ge 0$ and any distinct $i_1,\dotsc,i_k \in \{1,\dotsc,n\}$,
		\begin{equation*}
			W_2(\lim_{t \to \infty} \Lmc(X_{i_1}^n(t), \dotsc, X_{i_k}^n(t)), \, \lawlimit_{\frac{i_1}{n},\infty} \otimes \dotsb \otimes \lawlimit_{\frac{i_k}{n},\infty}) \le \frac{C\sqrt{k}}{\sqrt{n}}.
		\end{equation*}	
	\end{enumerate}
\end{Corollary}

\begin{proof}
%[Proof of Corollary \ref{cor:interchange-infinity}]
	Write
	\begin{equation*}
		W_2(\lawpre^n(\infty), \lawlimitmean(\infty)) \le W_2(\lawpre^n(\infty), \lawpre^n(t)) + W_2(\lawpre^n(t), \lawlimitmean(\infty)).
	\end{equation*}
	Using the convexity of $W_2^2(\cdot,\cdot)$, \eqref{eq:Wasserstein-joint} and \eqref{eq:exponential-ergodicity-n}, we have
	\begin{align*}
		W_2^2(\lawpre^n(\infty),\lawpre^n(t)) \le \frac{1}{n} \sum_{i=1}^n W_2^2(\lim_{s \to \infty}\Lmc(X_i^n(s)),\Lmc(X_i^n(t))) \le \frac{1}{n} W_2^2(\jointpre^n(\infty),\jointpre^n(t)) \le Ce^{-2\kappa t}
	\end{align*}
	for each $t \ge 0$.
	Combining these with Corollary \ref{cor:interchange} (resp.\ Corollary \ref{cor:interchange-Lipschitz}(a)) gives part (a) (resp.\ the first statement in part (b)).	
	The second statement in part (b) follows by taking $t \to \infty$ in Corollary \ref{cor:interchange-Lipschitz}(b).
\end{proof}

\section{Euler discretization}
\label{sec:Euler}

In this section we analyze the Euler discretization of the system \eqref{eq:system-n} with step size $h > 0$, namely, with $s_h := \lfl \frac{s}{h} \rfl h$,
\begin{align}
	X_i^{n,h}(t) & = \Xlimit_{\frac{i}{n}}(0) + \int_0^t \left( f(X_i^{n,h}(s_h)) + \frac{1}{n} \sum_{j=1}^n \xi_{ij}^n b(X_i^{n,h}(s_h),X_j^{n,h}(s_h)) \right) ds
	\notag \\
	& \quad + \sigma B_{\frac{i}{n}}(t), \quad i \in \{1,\dotsc,n\}, \: t \ge 0.
	\label{eq:system-n-h}
\end{align}

The following theorem shows the convergence of the Euler scheme, uniformly in time $t$ and the number of particles $n$.
The proof is given in Section \ref{sec:pf-Euler}.

\begin{Theorem}
	\label{thm:moment-convergence-uniform-Euler}
	There exist $C,h_0 \in (0,\infty)$ such that
	$$\sup_{n \in \Nmb} \max_{i=1,\dotsc,n} \sup_{t \ge 0} \Emb |X_i^{n,h}(t) - X_i^n(t)|^2 \le Ch, \quad \forall \, h \in (0,h_0).$$
\end{Theorem}

Theorem \ref{thm:moment-convergence-uniform-Euler} and Corollary \ref{cor:interchange} guarantee that the Euler scheme \eqref{eq:system-n-h} provides a good numerical approximation to the graphon particle system \eqref{eq:system} uniformly in time, as shown in the following corollary.
Let
\begin{align*}
	\emppre^{n,h}(t) := \frac{1}{n} \sum_{i=1}^n \delta_{X_i^{n,h}(t)}, 
%	\quad  \lawpre^{n,h}_i(t) := \Lmc(X_i^{n,h}(t)), 
	\quad \lawpre^{n,h}(t) := \frac{1}{n} \sum_{i=1}^n \Lmc(X_i^{n,h}(t)) = \Emb \emppre^{n,h}(t).
%	& \jointpre^{n,h}(t) := \Lmc(X_1^{n,h}(t),\dotsc,X_n^{n,h}(t)).
\end{align*}
	
\begin{Corollary}
	\label{cor:uniform-approximation-Euler}
	Suppose Conditions \ref{cond:G-continuous} and \ref{cond:G_n-step-graphon} hold.
	Then there exist $C,h_0 \in (0,\infty)$ such that
	\begin{equation*}
		\limsup_{n \to \infty} \sup_{t \ge 0} W_2(\lawpre^{n,h}(t),\lawlimitmean(t)) \le C\sqrt{h}, \quad \limsup_{n \to \infty} \sup_{t \ge 0} \Emb W_2(\emppre^{n,h}(t),\lawlimitmean(t)) \le C\sqrt{h},
	\end{equation*}
	and		
	\begin{equation*}
		\limsup_{n,t \to \infty} W_2(\lawpre^{n,h}(t),\lawlimitmean(\infty)) \le C\sqrt{h}, \quad \limsup_{n,t \to \infty} \Emb W_2(\emppre^{n,h}(t),\lawlimitmean(\infty)) \le C\sqrt{h},
	\end{equation*}
	for all $h \in (0,h_0)$.
\end{Corollary}	

\begin{proof}
%[Proof of Corollary \ref{cor:uniform-approximation-Euler}]
	Let $h_0 \in (0,\infty)$ be as in Theorem \ref{thm:moment-convergence-uniform-Euler}.
	Taking $\pi = \frac{1}{n} \sum_{i=1}^n \Lmc(X_i^{n,h}(t),X_i^n(t))$ as the coupling of $\lawpre^{n,h}(t)$ and $\lawpre^n(t)$ gives
	\begin{equation}
		\label{eq:uniform-approximation-Euler-1}
		\sup_{n \in \Nmb} \sup_{t \ge 0} W_2(\lawpre^{n,h}(t),\lawpre^n(t)) \le \sup_{n \in \Nmb} \sup_{t \ge 0} \left(\frac{1}{n} \sum_{i=1}^n \Emb |X_i^{n,h}(t)-X_i^n(t)|^2\right)^{1/2} \le C\sqrt{h},
	\end{equation}
	and taking $\pi = \frac{1}{n} \sum_{i=1}^n \delta_{(X_i^{n,h}(t),X_i^n(t))}$ as the coupling of $\emppre^{n,h}(t)$ and $\emppre^n(t)$ gives
	\begin{equation}
		\label{eq:uniform-approximation-Euler-2}
		\sup_{n \in \Nmb} \sup_{t \ge 0} \Emb W_2(\emppre^{n,h}(t),\emppre^n(t)) \le \sup_{n \in \Nmb} \sup_{t \ge 0} \left(\frac{1}{n} \sum_{i=1}^n \Emb |X_i^{n,h}(t)-X_i^n(t)|^2\right)^{1/2} \le C\sqrt{h},
	\end{equation}
	for all $h \in (0,h_0)$.
	Combining these with Theorem \ref{thm:moment-convergence-uniform}(b) (resp.\ Corollary \ref{cor:interchange}) gives the first (resp.\ second) statement. 
	This completes the proof.
\end{proof}

As stated in the second question in the introduction, we are also interested in the precise rate of convergence of the Euler scheme, as an approximation to the graphon particle system \eqref{eq:system} and its stationary distribution.
This is answered in the following corollary. 

\begin{Corollary}
	\label{cor:uniform-approximation-Euler-Lipschitz}
	Suppose Conditions \ref{cond:G-Lipschitz} and \ref{cond:G_n=G} hold.
	Then there exist $C,h_0 \in (0,\infty)$ such that the following hold.
	\begin{enumerate}[(a)]
	\item 
		For all $h \in (0,h_0)$, $n \in \Nmb$ and $t \ge 0$,
		\begin{equation*}
			W_2(\lawpre^{n,h}(t),\lawlimitmean(t)) \le C\left(\frac{1}{\sqrt{n}} + \sqrt{h} \right), \quad \Emb W_2(\emppre^{n,h}(t),\lawlimitmean(t)) \le C\left(a(n) + \sqrt{h}\right),
		\end{equation*}
		and
		\begin{equation*}
			W_2(\lawpre^{n,h}(t),\lawlimitmean(\infty)) \le C\left(\frac{1}{\sqrt{n}} + \sqrt{h} + e^{-\kappa t/2}\right), \:\Emb W_2(\emppre^{n,h}(t),\lawlimitmean(\infty)) \le C\left(a(n) + \sqrt{h} + e^{-\kappa t/2}\right).
		\end{equation*}
	\item
		For all $n,k \in \Nmb$, $t \ge 0$ and any distinct $i_1,\dotsc,i_k \in \{1,\dotsc,n\}$,
		\begin{equation*}
			W_2(\Lmc(X_{i_1}^{n,h}(t), \dotsc, X_{i_k}^{n,h}(t)), \, \lawlimit_{\frac{i_1}{n},t} \otimes \dotsb \otimes \lawlimit_{\frac{i_k}{n},t}) \le C\sqrt{k} \left( \frac{1}{\sqrt{n}} + \sqrt{h} \right),
		\end{equation*}	 
		and
		\begin{equation*}
			W_2(\Lmc(X_{i_1}^{n,h}(t), \dotsc, X_{i_k}^{n,h}(t)), \, \lawlimit_{\frac{i_1}{n},\infty} \otimes \dotsb \otimes \lawlimit_{\frac{i_k}{n},\infty}) \le C\sqrt{k} \left( \frac{1}{\sqrt{n}} + \sqrt{h} + e^{-\kappa t/2} \right).
		\end{equation*}		
	\end{enumerate}
\end{Corollary}	

\begin{proof}
%[Proof of Corollary \ref{cor:uniform-approximation-Euler-Lipschitz}]
	(a) Combining \eqref{eq:uniform-approximation-Euler-1}, \eqref{eq:uniform-approximation-Euler-2} and Theorem \ref{thm:moment-convergence-uniform-Lipschitz}(b) (resp.\ Corollary \ref{cor:interchange-Lipschitz}(a)) gives the first (resp.\ last) two statements.
	
	(b) Taking
%	\begin{equation*}
		$\pi = \Lmc \left( \left(X_{i_1}^{n,h}(t), \dotsc, X_{i_k}^{n,h}(t)\right), \left(X_{i_1}^n(t), \dotsc, X_{i_k}^n(t)\right) \right)$,
%	\end{equation*}
	we have
	\begin{align*}
		W_2(\Lmc(X_{i_1}^{n,h}(t), \dotsc, X_{i_k}^{n,h}(t)), \, \Lmc(X_{i_1}^n(t), \dotsc, X_{i_k}^n(t))) & \le \left( \sum_{j=1}^k \Emb |X_{i_j}^{n,h}(t) - X_{i_j}^n(t)|^2 \right)^{1/2} \\
		& \le C\sqrt{kh},
	\end{align*}
	where the last line uses Theorem \ref{thm:moment-convergence-uniform-Euler}.
	Combining this with Theorem \ref{thm:moment-convergence-uniform-Lipschitz}(c) (resp.\ Corollary \ref{cor:interchange-Lipschitz}(b)) gives the first (resp.\ second) statement.
	This completes the proof.
\end{proof}

\begin{Remark}
	The constant $C$ here could be made explicit, but we didn't explore that direction.
\end{Remark}

\section{Proofs}
\label{sec:pf}

We first present an elementary result that will be used in several later proofs.

\begin{Lemma}
	\label{lem:ODE}
	Let $y \colon [0,\infty) \to [0, \infty)$ be a non-negative differentiable function.
	Suppose 
	\begin{equation*}
		y(t)-y(r) \le -a_1 \int_r^t y(s)\,ds + a_2 \int_r^t \sqrt{y(s)} \,ds + \int_r^t a_3(s)\,ds, \quad \forall \, t > r \ge 0,
	\end{equation*} 
	for some $a_1 > 0$, $a_2 \in \Rmb$ and non-negative and continuous function $a_3$.
	Then 
	\begin{equation*}
		y(t) \le \max\left\{y(0), \left(\frac{a_2}{2a_1} + \sqrt{\frac{\sup_{0 \le s \le t} a_3(s)}{a_1} + \frac{a_2^2}{4a_1^2}}\right)^2 \right\}, \quad \forall \, t \ge 0.
	\end{equation*}
	In particular, $y(t) \le \max\left\{y(0), \frac{\sup_{0 \le s \le t} a_3(s)}{a_1} \right\}$ if $a_2=0$.
\end{Lemma}

\begin{proof}
	Fix $T \in (0,\infty)$.
	Since $y(t)$ is differentiable, we have
	\begin{equation*}
		y'(t) \le -a_1 y(t) + a_2 \sqrt{y(t)} + a_3(t) \le -a_1 \left( \sqrt{y(t)} - \frac{a_2}{2a_1} \right)^2 + \sup_{0 \le s \le T} a_3(s) + \frac{a_2^2}{4a_1}, \quad t \in [0,T].
	\end{equation*}
	Noting that the right hand side above is negative when 
	\begin{equation*}
		\sqrt{y(t)} > \frac{a_2}{2a_1} + \sqrt{\frac{\sup_{0 \le s \le T} a_3(s)}{a_1} + \frac{a_2^2}{4a_1^2}},
	\end{equation*}
	we have the desired result.
\end{proof}

\subsection{Proofs for Section \ref{sec:system-graphon}}
\label{sec:pf-system-graphon}

\begin{proof}[Proof of Proposition \ref{prop:finite-moment-uniform}]	
	Using It\^o's formula, Remark \ref{rmk:linear-growth}(b) and Proposition \ref{prop:well-posedness}(a), we have
	\begin{align*}
		& \Emb |\Xlimit_u(t)|^4 - \Emb |\Xlimit_u(0)|^4 \\
		& = \Emb \int_0^t 4|\Xlimit_u(s)|^2 \Xlimit_u(s) \cdot \left( f(\Xlimit_u(s)) + \int_I \int_{\Rmb^d} b(\Xlimit_u(s),x)G(u,v) \,\lawlimit_{v,s}(dx)\,dv \right) ds + Ct.
	\end{align*}
	Therefore the functions $$\alpha_u(t) := \Emb |\Xlimit_u(t)|^4, \quad \alpha(t) := \int_I \Emb |\Xlimit_v(t)|^4 \, dv = \int_I \alpha_v(t) \, dv$$ are differentiable, and
	\begin{align}
		& \Emb |\Xlimit_u(t)|^4 - \Emb |\Xlimit_u(r)|^4 \label{eq:finite-moment-uniform-1} \\
		& = \Emb \int_r^t 4|\Xlimit_u(s)|^2 \Xlimit_u(s) \cdot \left( f(\Xlimit_u(s)) + \int_I \int_{\Rmb^d} b(\Xlimit_u(s),x)G(u,v) \,\lawlimit_{v,s}(dx)\,dv \right) ds + C(t-r) \notag
	\end{align}		
	for all $t > r \ge 0$.
	From \eqref{eq:c0} we have
	\begin{equation*}
		x \cdot (f(x)-f(0)) \le -c_0 |x|^2
	\end{equation*}
	and hence for each $s \ge 0$,
	\begin{align}
		\Emb \left[ |\Xlimit_u(s)|^2 \Xlimit_u(s) \cdot f(\Xlimit_u(s)) \right]
		& \le \Emb \left[ -c_0 |\Xlimit_u(s)|^4 + |f(0)||\Xlimit_u(s)|^3 \right] \notag \\
		& \le -c_0 \alpha_u(s) + \frac{c_0-2K_b}{4} \alpha_u(s) + C,  \label{eq:finite-moment-uniform-2}
	\end{align}
	where the last line uses Young's inequality and \eqref{eq:kappa}.
	For the rest of integrand in \eqref{eq:finite-moment-uniform-1}, using the Lipschitz property of $b$ we have
	\begin{align*}
		& \Emb \left[ |\Xlimit_u(s)|^2 \Xlimit_u(s) \cdot \int_I \int_{\Rmb^d} b(\Xlimit_u(s),x)G(u,v) \,\lawlimit_{v,s}(dx)\,dv \right] \\
		& \le \int_I \int_{\Rmb^d} \Emb \left[ |\Xlimit_u(s)|^3 \left( |b(0,0)| + K_b|\Xlimit_u(s)| + K_b|x| \right) \right] \lawlimit_{v,s}(dx)\,dv \\
		& \le C \alpha_u^{3/4}(s) + K_b \alpha_u(s) + K_b \alpha_u^{3/4}(s) \int_I \alpha_v^{1/4}(s) \, dv \\
		& \le \frac{c_0-2K_b}{4} \alpha_u(s) + C + K_b \alpha_u(s) + \frac{3}{4}K_b \alpha_u(s) + \frac{1}{4}K_b \alpha(s),
	\end{align*}	
	where the third line uses Jensen's inequality and the last line uses Young's inequality and \eqref{eq:kappa}.
	Combining this with \eqref{eq:finite-moment-uniform-1} and \eqref{eq:finite-moment-uniform-2} gives
	\begin{align}
		& \alpha_u(t) - \alpha_u(r) \le -(2c_0-3K_b) \int_r^t \alpha_u(s) \, ds + K_b \int_r^t \alpha(s) \,ds + C(t-r). \label{eq:finite-moment-uniform-3}
	\end{align}
	Integrating over $u \in I$ gives
	\begin{align*}
		& \alpha(t) - \alpha(r) \le -2(c_0-2K_b) \int_r^t \alpha(s) \, ds + C(t-r).
	\end{align*}	
	Since the function $\alpha(t)$ is non-negative and differentiable, using Lemma \ref{lem:ODE} (with $a_1=2(c_0-2K_b)$, $a_2=0$, $a_3=C$) we have
%	functions $\alpha_u(t) := \Emb |\Xlimit_u(t)|^4$ and $\alpha(t) := \int_I \Emb |\Xlimit_v(t)|^4 \, dv = \int_I \alpha_v(t) \, dv$ are differentiable, we have
%	\begin{equation*}
%%		\label{eq:finite-moment-uniform-2}
%		\alpha'(t) \le -2(c_0-2K_b) \alpha(t) + C.
%	\end{equation*}
%	Noting that the right hand side above is negative when 
%	$\alpha(t) > \frac{C}{2(c_0-2K_b)}$,
%	we must have
%	\begin{equation*}
		$\alpha(t) \le C$.
%	\end{equation*}
	Applying this to \eqref{eq:finite-moment-uniform-3} gives
	\begin{equation*}
		\alpha_u(t) - \alpha_u(r) \le -(2c_0-3K_b) \int_r^t \alpha_u(s)\,ds + C(t-r).
	\end{equation*}	
	Since the function $\alpha_u(t)$ is non-negative and differentiable, using Lemma \ref{lem:ODE} again we have
%	and hence
%	\begin{equation*}
%		\alpha_u'(t) \le -(2c_0-3K_b) \alpha_u(t) + C.
%	\end{equation*}	
%	Again the right hand side is negative when $\alpha_u(t) > \frac{C}{2c_0-3K_b}$.
%	Therefore we must have
%	\begin{equation*}
		$\alpha_u(t) \le C$,
%	\end{equation*}	
	uniformly in $u \in I$.
	This completes the proof.
\end{proof}

\begin{proof}[Proof of Theorem \ref{thm:exponential-ergodicity}]
	(a) From Proposition \ref{prop:finite-moment-uniform} we see that the quantity $\kappa_1 = \sup_{u \in I} \sup_{t \ge 0} \Emb |\Xlimit_u(t)|^2$ is finite.
	Let 
	$$A := \left\{ \eta := (\eta_u : u \in I) \in [\Pmc(\Rmb^d)]^I : \sup_{u \in I} \int_{\Rmb^d} |x|^2\,\eta_u(dx) \le \kappa_1 \right\}.$$  	
	Recall $\kappa$ in \eqref{eq:kappa} and the process $\Ybar_u^\eta$ in \eqref{eq:Ybar_eta}.
	
	We claim that
	\begin{equation}
		\label{eq:exponential-ergodicity-claim}
		\sup_{u \in I} W_2^2((P_t\eta)_u,(P_t\etatil)_u) \le \frac{c_0-K_b}{\kappa} e^{-\kappa t} \sup_{u \in I} W_2^2(\eta_u,\etatil_u)
	\end{equation}
	for any $\eta,\etatil \in A$ and $t \ge 0$.
	To see this, by It\^o's formula, we have
	\begin{align*}
		& e^{\kappa t} \Emb |\Ybar_u^\eta(t)-\Ybar_u^\etatil(t)|^2 - \Emb |\Ybar_u^\eta(0)-\Ybar_u^\etatil(0)|^2 \\
		& = \int_0^t e^{\kappa s} \Emb \left[ 2 \left(\Ybar_u^\eta(s)-\Ybar_u^\etatil(s)\right) \cdot \left( f(\Ybar_u^\eta(s)) - f(\Ybar_u^\etatil(s)) \right. \right. \\
		& \qquad \left. \left. + \int_I \left( \int_{\Rmb^d} b(\Ybar_u^\eta(s),x) \,\lawlimit_{v,s}^\eta(dx) - \int_{\Rmb^d} b(\Ybar_u^\etatil(s),x) \,\lawlimit_{v,s}^\etatil(dx) \right) G(u,v) \,dv \right) \right] ds \\
		& \quad + \int_0^t \kappa e^{\kappa s} \Emb |\Ybar_u^\eta(s)-\Ybar_u^\etatil(s)|^2 \,ds.
	\end{align*}	
	Using \eqref{eq:c0} we have
	\begin{align*}
		\Emb \left[ \left(\Ybar_u^\eta(s)-\Ybar_u^\etatil(s)\right) \cdot \left( f(\Ybar_u^\eta(s)) - f(\Ybar_u^\etatil(s)) \right) \right] \le -c_0 \Emb |\Ybar_u^\eta(s)-\Ybar_u^\etatil(s)|^2.
	\end{align*}
	By adding and subtracting terms, we have
	\begin{align*}
		& \Emb \left[ \left(\Ybar_u^\eta(s)-\Ybar_u^\etatil(s)\right) \cdot \int_I \left( \int_{\Rmb^d} b(\Ybar_u^\eta(s),x) \,\lawlimit_{v,s}^\eta(dx) - \int_{\Rmb^d} b(\Ybar_u^\etatil(s),x) \,\lawlimit_{v,s}^\etatil(dx) \right) G(u,v) \,dv \right] \\
		& \le \Emb \left[ \left(\Ybar_u^\eta(s)-\Ybar_u^\etatil(s)\right) \cdot \int_I \int_{\Rmb^d} \left( b(\Ybar_u^\eta(s),x) - b(\Ybar_u^\etatil(s),x) \right) \lawlimit_{v,s}^\eta(dx) \,G(u,v) \,dv \right] \\
		& \quad + \Emb \left[ \left(\Ybar_u^\eta(s)-\Ybar_u^\etatil(s)\right) \cdot \int_I \int_{\Rmb^d} b(\Ybar_u^\etatil(s),x) \left(\lawlimit_{v,s}^\eta(dx)-\lawlimit_{v,s}^\etatil(dx)\right) G(u,v) \,dv \right] \\
		& \le K_b \Emb |\Ybar_u^\eta(s)-\Ybar_u^\etatil(s)|^2 + K_b \Emb |\Ybar_u^\eta(s)-\Ybar_u^\etatil(s)| \int_I W_2(\lawlimit_{v,s}^\eta, \lawlimit_{v,s}^\etatil) \,dv \\
		& \le \frac{3K_b}{2} \Emb |\Ybar_u^\eta(s)-\Ybar_u^\etatil(s)|^2 + \frac{K_b}{2} \int_I W_2^2(\lawlimit_{v,s}^\eta, \lawlimit_{v,s}^\etatil) \,dv,
	\end{align*}
	where the fourth line uses the Lipschitz property of $b$ and \eqref{eq:Wasserstein-duality} and the last line uses Young's inequality and Jensen's inequality.
	Combining above three displays gives
	\begin{align}
		& e^{\kappa t} \Emb |\Ybar_u^\eta(t)-\Ybar_u^\etatil(t)|^2- \Emb |\Ybar_u^\eta(0)-\Ybar_u^\etatil(0)|^2 \notag \\
		& \quad \le -(2c_0-3K_b-\kappa) \int_0^t e^{\kappa s} \Emb |\Ybar_u^\eta(s)-\Ybar_u^\etatil(s)|^2 \,ds + K_b \int_0^t e^{\kappa s} \int_I W_2^2(\lawlimit_{v,s}^\eta, \lawlimit_{v,s}^\etatil) \,dv\,ds. \label{eq:exponential-ergodicity-1}
	\end{align}
	Since the function $t \mapsto e^{\kappa t} \Emb |\Ybar_u^\eta(t)-\Ybar_u^\etatil(t)|^2$ is non-negative and differentiable, using Lemma \ref{lem:ODE} (with $a_1=2c_0-3K_b-\kappa$, $a_2=0$, $a_3(s)=K_b e^{\kappa s} \int_I W_2^2(\lawlimit_{v,s}^\eta, \lawlimit_{v,s}^\etatil) \,dv$) we have
	\begin{align*}
		e^{\kappa t} \Emb |\Ybar_u^\eta(t)-\Ybar_u^\etatil(t)|^2 & \le \max \left\{ \Emb |\Ybar_u^\eta(0)-\Ybar_u^\etatil(0)|^2, \frac{K_b \sup_{0 \le s \le t} e^{\kappa s} \int_I W_2^2(\lawlimit_{v,s}^\eta, \lawlimit_{v,s}^\etatil) \,dv}{2c_0-3K_b-\kappa} \right\} \\
		& \le \Emb |\Ybar_u^\eta(0)-\Ybar_u^\etatil(0)|^2 + \frac{K_b}{2c_0-3K_b-\kappa} \sup_{0 \le s \le t} e^{\kappa s} \sup_{v \in I} W_2^2(\lawlimit_{v,s}^\eta, \lawlimit_{v,s}^\etatil).
	\end{align*}
	Taking the infimum over the joint distribution of $(\Ybar_u^\eta(0),\Ybar_u^\etatil(0))$ gives
	\begin{align*}
		e^{\kappa t} W_2^2((P_t\eta)_u,(P_t\etatil)_u) & \le e^{\kappa t} \Emb |\Ybar_u^\eta(t)-\Ybar_u^\etatil(t)|^2 \\
		& \le W_2^2(\eta_u,\etatil_u) + \frac{K_b}{2c_0-3K_b-\kappa} \sup_{0 \le s \le t} e^{\kappa s} \sup_{v \in I} W_2^2(\lawlimit_{v,s}^\eta, \lawlimit_{v,s}^\etatil).
	\end{align*}
	Taking the supremum over $u \in I$ and the time interval $[0,t]$ gives
	\begin{equation*}
		\sup_{0 \le s \le t} e^{\kappa s} \sup_{u \in I} W_2^2((P_s\eta)_u,(P_s\etatil)_u) \le \sup_{u \in I} W_2^2(\eta_u,\etatil_u) + \frac{K_b}{2c_0-3K_b-\kappa} \sup_{0 \le s \le t} e^{\kappa s} \sup_{u \in I} W_2^2(\lawlimit_{u,s}^\eta, \lawlimit_{u,s}^\etatil).
	\end{equation*}
	Since $\kappa = c_0-2K_b>0$, by rearranging terms we have
	\begin{equation*}
		\sup_{0 \le s \le t} e^{\kappa s} \sup_{u \in I} W_2^2((P_s\eta)_u,(P_s\etatil)_u) \le \frac{c_0-K_b}{c_0-2K_b} \sup_{u \in I} W_2^2(\eta_u,\etatil_u).
	\end{equation*}
	This gives the claim \eqref{eq:exponential-ergodicity-claim}.
	
	Note that $\lawlimit(t) := (\lawlimit_{u,t} : u \in I) \in A$ and $\lawlimit(t) = P_t\lawlimit(0)$ for each $t \ge 0$ by Propositions \ref{prop:well-posedness}(a) and \ref{prop:finite-moment-uniform}.
	It then follows from \eqref{eq:exponential-ergodicity-claim} that
	\begin{align}
		W_2^2(\lawlimit_{u,t+s},\lawlimit_{u,t}) & = W_2^2((P_t\lawlimit(s))_u,(P_t\lawlimit(0))_u) \notag \\
		& \le \frac{c_0-K_b}{\kappa}e^{-\kappa t} W_2^2(\lawlimit_{u,s},\lawlimit_{u,0}) \notag \\
		& \le 4\kappa_1 \frac{c_0-K_b}{\kappa}e^{-\kappa t}. \label{eq:exponential-ergodicity-2}
	\end{align} 
	This means that $\lawlimit_{u,t}$ is a $W_2$-Cauchy family when $t \to \infty$.
	So there exists a probability measure $\lawlimit_{u,\infty} \in \Pmc(\Rd)$ such that
	\begin{equation}
		\label{eq:exponential-ergodicity-3}
		\lim_{t \to \infty} W_2(\lawlimit_{u,t},\lawlimit_{u,\infty}) = 0.
	\end{equation}
	In fact, taking $s \to \infty$ in \eqref{eq:exponential-ergodicity-2} gives $W_2(\lawlimit_{u,t},\lawlimit_{u,\infty}) \le \sqrt{4\kappa_1\frac{c_0-K_b}{\kappa}} e^{-\kappa t/2}$, uniformly in $u \in I$.
	This gives \eqref{eq:exponential-ergodicity}.
	Since $W_2^2(\cdot,\cdot)$ is convex, \eqref{eq:exponential-ergodicity-average} follows from \eqref{eq:exponential-ergodicity} and \eqref{eq:lawlimitmean}.	
	
	(b) Next we argue that $\lawlimit(\infty)$ is invariant with respect to $P_t$.
	Noting that $\lawlimit(\infty) \in A$,
%	 and using Proposition \ref{prop:well-posedness}(a) we have $P_t\lawlimit(\infty) \in A$ for each $t \ge 0$, with possibly a larger $C_0$.
	we can apply \eqref{eq:exponential-ergodicity-claim} and use \eqref{eq:exponential-ergodicity-3} to get
	\begin{equation*}
		\limsup_{s \to \infty} W_2((P_t\lawlimit(\infty))_u,\lawlimit_{u,t+s}) \le \limsup_{s \to \infty} \sqrt{\frac{c_0-K_b}{\kappa}} e^{-\kappa t/2} W_2(\lawlimit_{u,\infty},\lawlimit_{u,s}) = 0
	\end{equation*}
	and
	\begin{equation*}
		\limsup_{s \to \infty} W_2(\lawlimit_{u,t+s},\lawlimit_{u,s}) \le \limsup_{s \to \infty} \sqrt{\frac{c_0-K_b}{\kappa}} e^{-\kappa s/2} W_2(\lawlimit_{u,t},\lawlimit_{u,0}) = 0.
	\end{equation*}
	Combining these two with \eqref{eq:exponential-ergodicity-3} gives
	\begin{align*}
		W_2((P_t\lawlimit(\infty))_u,\lawlimit_{u,\infty})
		& \le \limsup_{s \to \infty} W_2((P_t\lawlimit(\infty))_u,\lawlimit_{u,t+s}) \\
		& \quad + \limsup_{s \to \infty} W_2(\lawlimit_{u,t+s},\lawlimit_{u,s}) + \limsup_{s \to \infty} W_2(\lawlimit_{u,s},\lawlimit_{u,\infty}) \\
		& = 0.
	\end{align*} 
	This gives part (b).
	
	(c) 
	Fix $u_1,u_2 \in I$. 
	Consider the following diffusions:
	\begin{align*}
		\Xtil_{u_1}(t) & = \Xtil_{u_1}(0) + \int_0^t \left( f(\Xtil_{u_1}(s)) + \int_I \int_{\Rmb^d} b(\Xtil_{u_1}(s),x)G(u_1,v) \,\lawlimit_{v,s}(dx)\,dv\right)ds + \sigma B(t), \\
		\Xtil_{u_2}(t) & = \Xtil_{u_2}(0) + \int_0^t \left( f(\Xtil_{u_2}(s)) + \int_I \int_{\Rmb^d} b(\Xtil_{u_2}(s),x)G(u_2,v) \,\lawlimit_{v,s}(dx)\,dv\right)ds + \sigma B(t).
	\end{align*}
	Here $B$ is a $d$-dimensional Brownian motion independent of $\{\Xtil_{u_1}(0),\Xtil_{u_2}(0)\}$, $\Lmc(\Xtil_{u_1}(0))=\lawlimit_{u_1,0}$, $\Lmc(\Xtil_{u_2}(0))=\lawlimit_{u_2,0}$, but $\Xtil_{u_1}(0)$ and $\Xtil_{u_2}(0)$ may not be independent.
	From the uniqueness property in Proposition \ref{prop:well-posedness}(a) we have $\Lmc(\Xtil_{u_1})=\lawlimit_{u_1}$ and $\Lmc(\Xtil_{u_2})=\lawlimit_{u_2}$.
	Using It\^o's formula we have
	\begin{align*}
		& \Emb |\Xtil_{u_1}(t)-\Xtil_{u_2}(t)|^2 - \Emb |\Xtil_{u_1}(r)-\Xtil_{u_2}(r)|^2 \\
		& = \Emb \int_r^t 2 (\Xtil_{u_1}(s)-\Xtil_{u_2}(s)) \cdot \left( f(\Xtil_{u_1}(s)) - f(\Xtil_{u_2}(s)) \right. \\
		& \qquad \left. + \int_I \int_{\Rmb^d} \left( b(\Xtil_{u_1}(s),x)G(u_1,v) - b(\Xtil_{u_2}(s),x)G(u_2,v) \right) \lawlimit_{v,s}(dx)\,dv \right) ds.
	\end{align*}	
	For each $s \ge 0$, from \eqref{eq:c0} we have
	\begin{align*}
		& \Emb \left[ (\Xtil_{u_1}(s)-\Xtil_{u_2}(s)) \cdot \left(f(\Xtil_{u_1}(s)) - f(\Xtil_{u_2}(s)) \right) \right] \le -c_0\Emb |\Xtil_{u_1}(s)-\Xtil_{u_2}(s)|^2.
	\end{align*}
	By adding and subtracting terms we get
	\begin{align*}
		& \Emb \left| (\Xtil_{u_1}(s)-\Xtil_{u_2}(s)) \cdot \int_I \int_{\Rmb^d} \left( b(\Xtil_{u_1}(s),x)G(u_1,v) - b(\Xtil_{u_2}(s),x)G(u_2,v) \right) \lawlimit_{v,s}(dx)\,dv \right| \\
		& \le \Emb \left| (\Xtil_{u_1}(s)-\Xtil_{u_2}(s)) \cdot \int_I \int_{\Rmb^d} \left( b(\Xtil_{u_1}(s),x) - b(\Xtil_{u_2}(s),x) \right) G(u_1,v)\, \lawlimit_{v,s}(dx)\,dv \right| \\
		& \quad + \Emb \left| (\Xtil_{u_1}(s)-\Xtil_{u_2}(s)) \cdot \int_I \int_{\Rmb^d} b(\Xtil_{u_2}(s),x) \left(G(u_1,v) - G(u_2,v) \right) \lawlimit_{v,s}(dx)\,dv \right|.
	\end{align*}
	For the first term on the right hand side, it follows from the Lipschitz property of $b$ that
	\begin{align*}
		& \Emb \left| (\Xtil_{u_1}(s)-\Xtil_{u_2}(s)) \cdot \int_I \int_{\Rmb^d} \left( b(\Xtil_{u_1}(s),x) - b(\Xtil_{u_2}(s),x) \right) G(u_1,v)\, \lawlimit_{v,s}(dx)\,dv \right| \\
		& \le K_b\Emb |\Xtil_{u_1}(s)-\Xtil_{u_2}(s)|^2.
	\end{align*}
	For the other term, using the Cauchy-Schwarz inequality, Young's inequality, the Lipschitz property of $b$ and Proposition \ref{prop:finite-moment-uniform} we have
	\begin{align*}
		& \Emb \left| (\Xtil_{u_1}(s)-\Xtil_{u_2}(s)) \cdot \int_I \int_{\Rmb^d} b(\Xtil_{u_2}(s),x) \left(G(u_1,v) - G(u_2,v) \right) \lawlimit_{v,s}(dx)\,dv \right| \\
		& \le K_b \Emb |\Xtil_{u_1}(s)-\Xtil_{u_2}(s)|^2 + \frac{1}{4K_b} \Emb \left( \int_I \int_{\Rmb^d} |b(\Xtil_{u_2}(s),x)| |G(u_1,v) - G(u_2,v)| \, \lawlimit_{v,s}(dx)\,dv \right)^2 \\
		& \le K_b \Emb |\Xtil_{u_1}(s)-\Xtil_{u_2}(s)|^2 + C \Emb \left( \int_I \left( 1 + |\Xtil_{u_2}(s)| \right) |G(u_1,v) - G(u_2,v)| \,dv \right)^2 \\
		& \le K_b\Emb |\Xtil_{u_1}(s)-\Xtil_{u_2}(s)|^2 + C \left( \int_I |G(u_1,v)-G(u_2,v)| \,dv \right)^2.
	\end{align*}	
	Combining above five displays gives
	\begin{align*}
		& \Emb |\Xtil_{u_1}(t)-\Xtil_{u_2}(t)|^2 - \Emb |\Xtil_{u_1}(r)-\Xtil_{u_2}(r)|^2 \\
		& \le -2(c_0-2K_b) \int_r^t \Emb |\Xtil_{u_1}(s)-\Xtil_{u_2}(s)|^2 \,ds + C(t-r) \left( \int_I |G(u_1,v)-G(u_2,v)| \,dv \right)^2.
	\end{align*}
	Since the function $t \mapsto \Emb |\Xtil_{u_1}(t)-\Xtil_{u_2}(t)|^2$ is non-negative and differentiable, using Lemma \ref{lem:ODE} (with $a_1=2(c_0-2K_b)$, $a_2=0$, $a_3=C\left( \int_I |G(u_1,v)-G(u_2,v)| \,dv \right)^2$) we have
	$$\Emb |\Xtil_{u_1}(t)-\Xtil_{u_2}(t)|^2 \le \max \left\{ \Emb |\Xtil_{u_1}(0)-\Xtil_{u_2}(0)|^2, C \left( \int_I |G(u_1,v)-G(u_2,v)| \,dv \right)^2 \right\}.$$
	Taking the infimum over the joint distribution of $\Xtil_{u_1}(0)$ and $\Xtil_{u_2}(0)$ gives part (c) and completes the proof.
\end{proof}

\subsection{Proofs for Section \ref{sec:system-n}}
\label{sec:pf-system-n}

\begin{proof}[Proof of Proposition \ref{prop:finite-moment-uniform-n}]
	Fix $n \in \Nmb$ and $z=(z_{ij}=z_{ji})_{i,j=1}^n \in [0,1]^{n^2}$.
	Using the Lipschitz property of $f,b$ and a standard argument one has
	\begin{equation*}
		\max_{i=1,\dotsc,n} \sup_{0 \le t \le T} \Emb^{n,z} |X_i^n(t)|^2 < \infty \;\: a.s., \quad \forall \, T \in (0,\infty).
	\end{equation*}
	Using this and It\^o's formula, we have
	\begin{align*}
		& \Emb^{n,z} |X_i^n(t)|^2 - \Emb^{n,z} |X_i^n(r)|^2 \\
		& = \Emb^{n,z} \int_r^t 2X_i^n(s) \cdot \left( f(X_i^n(s)) + \frac{1}{n} \sum_{j=1}^n z_{ij}b(X_i^n(s),X_j^n(s)) \right) ds + C(t-r),
	\end{align*}
	and hence the functions
	\begin{align*}
		\alpha_i^{n,z}(t) := \Emb^{n,z} |X_i^n(t)|^2, \quad \alpha^{n,z}(t) := \frac{1}{n} \sum_{i=1}^n \Emb^{n,z} |X_i^n(t)|^2
	\end{align*}	
	are differentiable.
%	From \eqref{eq:c0} we have
%	$$x \cdot (f(x)-f(0)) \le -c_0 |x|^2.$$
	For each $s \ge 0$, using \eqref{eq:c0} and the Lipschitz property of $b$ we have
	\begin{align*}
		& \Emb^{n,z} \left[ X_i^n(s) \cdot \left( f(X_i^n(s)) + \frac{1}{n} \sum_{j=1}^n z_{ij}b(X_i^n(s),X_j^n(s)) \right) \right] \\
		& \le \Emb^{n,z} \left[ -c_0 |X_i^n(s)|^2 + C|X_i^n(s)| + \frac{1}{n} \sum_{j=1}^n |X_i^n(s)| \left( C + K_b|X_i^n(s)| + K_b|X_j^n(s)| \right)  \right] \\
		& \le -c_0 \alpha_i^{n,z}(s) + \frac{c_0-2K_b}{2} \alpha_i^{n,z}(s) + C + K_b \alpha_i^{n,z}(s) + \frac{K_b}{2} \alpha_i^{n,z}(s) + \frac{K_b}{2} \alpha^{n,z}(s),
	\end{align*}
	where the last line uses Young's inequality and \eqref{eq:kappa}.	
	Therefore
	\begin{align}
		& \alpha_i^{n,z}(t) - \alpha_i^{n,z}(r) \le -\left(c_0-K_b\right) \int_r^t \alpha_i^{n,z}(s) \, ds + K_b \int_r^t \alpha^{n,z}(s) \,ds + C(t-r). \label{eq:finite-moment-uniform-n-1}
	\end{align}
	Taking the average over $i=1,\dotsc,n$ gives
	\begin{align*}
		\alpha^{n,z}(t) - \alpha^{n,z}(r) & \le -(c_0-2K_b) \int_r^t \alpha^{n,z}(s) \, ds + C(t-r). 
%		\label{eq:finite-moment-uniform-n-2}
	\end{align*}	
%	Therefore
%	\begin{equation*}
%		\frac{d}{dt} \alpha^{n,z}(t) \le -(c_0-2K_b) \alpha^{n,z}(t) + C.
%	\end{equation*}
%	Noting that the right hand side is negative when 
%	$\alpha^{n,z}(t) > \frac{C}{c_0-2K_b}$, 
%	we must have
	Since the function $\alpha^{n,z}(t)$ is non-negative and differentiable, using Lemma \ref{lem:ODE} (with $a_1=c_0-2K_b$, $a_2=0$, $a_3=C$) we have	
%	\begin{equation*}
		$\alpha^{n,z}(t) \le C$.
%	\end{equation*}
	From this and \eqref{eq:finite-moment-uniform-n-1} we further have
	\begin{equation*}
		\alpha_i^{n,z}(t) - \alpha_i^{n,z}(r) \le -(c_0-K_b) \int_r^t \alpha_i^{n,z}(s)\,ds + C(t-r).
	\end{equation*}	
	Since the function $\alpha_i^{n,z}(t)$ is non-negative and differentiable, using Lemma \ref{lem:ODE} again we have	
%	and hence
%	\begin{equation*}
%		\frac{d}{dt} \alpha_i^{n,z}(t) \le -(c_0-K_b) \alpha_i^{n,z}(t) + C.
%	\end{equation*}	
%	Again the right hand side is negative when 
%	$\alpha_i^{n,z}(t) > \frac{C}{c_0-K_b}.$ 
%	Therefore we must have
%	\begin{equation*}
		$\alpha_i^{n,z}(t) \le C$,
%	\end{equation*}	
	uniformly in $n \in \Nmb$, $i=1,\dotsc,n$ and $t \ge 0$.
	This completes the proof.	
\end{proof}

\begin{proof}[Proof of Theorem \ref{thm:exponential-ergodicity-n}]
	(a) Recall Proposition \ref{prop:finite-moment-uniform-n} and $\kappa_2$ therein.
	Fix $n \in \Nmb$ and $z=(z_{ij}=z_{ji})_{i,j=1}^n \in [0,1]^{n^2}$.
	Let
	$$A_n :=\left\{ \eta \in \Pmc((\Rmb^d)^n) : \max_{i=1,\dotsc,n} \int_{(\Rmb^d)^n} |x_i|^2 \,\eta(dx) \le \kappa_2 \right\}.$$
	Recall the process $Y^{n,z,\eta}$ in \eqref{eq:Y_eta-n}.
	 
	We claim that
	\begin{equation}
		\label{eq:exponential-ergodicity-n-claim}
		W_2^2(P_t^{n,z}\eta,P_t^{n,z}\etatil) \le W_2^2(\eta,\etatil) e^{-2\kappa t}
	\end{equation}
	for any $\eta,\etatil \in A_n$ and $t \ge 0$.
	To see this, by It\^o's formula, we have
	\begin{align*}
		& e^{2\kappa t} \sum_{i=1}^n \Emb^{n,z} |Y_i^{n,z,\eta}(t)-Y_i^{n,z,\etatil}(t)|^2
		- \sum_{i=1}^n \Emb^{n,z} |Y_i^{n,z,\eta}(0)-Y_i^{n,z,\etatil}(0)|^2 \\
		& = \int_0^t e^{2\kappa s} \sum_{i=1}^n \Emb^{n,z} \left[ 2 \left( Y_i^{n,z,\eta}(s)-Y_i^{n,z,\etatil}(s) \right) \cdot \left( f(Y_i^{n,z,\eta}(s)) - f(Y_i^{n,z,\etatil}(s)) \right. \right. \\
		& \qquad \left. \left. + \frac{1}{n} \sum_{j=1}^n z_{ij} \left( b(Y_i^{n,z,\eta}(s),Y_j^{n,z,\eta}(s)) - b(Y_i^{n,z,\etatil}(s),Y_j^{n,z,\etatil}(s)) \right) \right) \right] ds \\
		& \quad + \int_0^t 2\kappa e^{2\kappa s} \sum_{i=1}^n \Emb^{n,z} |Y_i^{n,z,\eta}(s)-Y_i^{n,z,\etatil}(s)|^2 \, ds.
	\end{align*}	
	Using \eqref{eq:c0} we have
	\begin{align*}
		& \sum_{i=1}^n \Emb^{n,z} \left[ \left( Y_i^{n,z,\eta}(s)-Y_i^{n,z,\etatil}(s) \right) \cdot \left( f(Y_i^{n,z,\eta}(s)) - f(Y_i^{n,z,\etatil}(s)) \right) \right] \\
		& \le -c_0 \sum_{i=1}^n \Emb^{n,z} |Y_i^{n,z,\eta}(s)-Y_i^{n,z,\etatil}(s)|^2.
	\end{align*}
	Using the Cauchy-Schwarz inequality, the Lipschitz property of $b$ and Young's inequality we have
	\begin{align*}
		& \sum_{i=1}^n \Emb^{n,z} \left[ \left( Y_i^{n,z,\eta}(s)-Y_i^{n,z,\etatil}(s) \right) \cdot \frac{1}{n} \sum_{j=1}^n z_{ij} \left( b(Y_i^{n,z,\eta}(s),Y_j^{n,z,\eta}(s)) - b(Y_i^{n,z,\etatil}(s),Y_j^{n,z,\etatil}(s)) \right) \right] \\
		& \le \sum_{i=1}^n \Emb^{n,z} \left[ \left| Y_i^{n,z,\eta}(s)-Y_i^{n,z,\etatil}(s) \right| \cdot \frac{K_b}{n} \sum_{j=1}^n \left( \left|Y_i^{n,z,\eta}(s)-Y_i^{n,z,\etatil}(s)\right| + \left|Y_j^{n,z,\eta}(s)-Y_j^{n,z,\etatil}(s)\right| \right) \right] \\
		& \le 2K_b \sum_{i=1}^n \Emb^{n,z} |Y_i^{n,z,\eta}(s)-Y_i^{n,z,\etatil}(s)|^2.
	\end{align*}
	Combining above three estimates with the definition of $\kappa$ in \eqref{eq:kappa} gives
	$$e^{2\kappa t} \sum_{i=1}^n \Emb^{n,z} |Y_i^{n,z,\eta}(t)-Y_i^{n,z,\etatil}(t)|^2 \le \sum_{i=1}^n \Emb^{n,z} |Y_i^{n,z,\eta}(0)-Y_i^{n,z,\etatil}(0)|^2.$$
	Therefore
	$$W_2^2(P_t^{n,z}\eta,P_t^{n,z}\etatil) \le \sum_{i=1}^n \Emb^{n,z} |Y_i^{n,z,\eta}(t)-Y_i^{n,z,\etatil}(t)|^2 \le e^{-2\kappa t} \sum_{i=1}^n \Emb^{n,z} |Y_i^{n,z,\eta}(0)-Y_i^{n,z,\etatil}(0)|^2.$$
	Taking the infimum over the joint distribution of $(Y^{n,z,\eta}(0),Y^{n,z,\etatil}(0))$ gives the claim \eqref{eq:exponential-ergodicity-n-claim}.
	
	Note that $\jointpre^{n,z}(t) = P_t \jointpre^{n,z}(0) \in A_n$ for each $t \ge 0$ by Proposition \ref{prop:finite-moment-uniform-n}.
	Therefore 
	$$W_2^2(\jointpre^{n,z}(t),\jointpre^{n,z}(0)) \le \sum_{i=1}^n \Emb^{n,z}|X_i^n(t)-X_i^n(0)|^2 \le 2 \sum_{i=1}^n \Emb^{n,z}[|X_i^n(t)|^2 + |X_i^n(0)|^2] \le 4n\kappa_2.$$
	It then follows from \eqref{eq:exponential-ergodicity-n-claim} that
	\begin{align}
		W_2^2(\jointpre^{n,z}(t+s),\jointpre^{n,z}(t)) & = W_2^2(P_t^{n,z}\jointpre^{n,z}(s),P_t^{n,z}\jointpre^{n,z}(0)) \notag \\
		& \le W_2^2(\jointpre^{n,z}(s),\jointpre^{n,z}(0)) e^{-2\kappa t} \notag \\
		& \le 4n\kappa_2 e^{-2\kappa t}. \label{eq:exponential-ergodicity-n-1}
	\end{align} 
	This means that $\jointpre^{n,z}(t)$ is a $W_2$-Cauchy family when $t \to \infty$.
	So there exists a probability measure $\jointpre^{n,z}(\infty) \in \Pmc((\Rd)^n)$ such that
	\begin{equation}
		\label{eq:exponential-ergodicity-n-2}
		\lim_{t \to \infty} W_2(\jointpre^{n,z}(t),\jointpre^{n,z}(\infty)) = 0.
	\end{equation}
	In fact, taking $s \to \infty$ in \eqref{eq:exponential-ergodicity-n-1} gives 
	$$\frac{1}{\sqrt{n}} W_2(\jointpre^{n,z}(t),\jointpre^{n,z}(\infty)) \le \sqrt{4\kappa_2} e^{-\kappa t},$$
	which gives \eqref{eq:exponential-ergodicity-n-z}.
	From this and \eqref{eq:theta-infinity} we have
	$$\frac{1}{\sqrt{n}} W_2(\jointpre^n(t),\jointpre^n(\infty)) \le \sup_{z \in [0,1]^{n^2}} \frac{1}{n} W_2(\jointpre^{n,z}(t),\jointpre^{n,z}(\infty)) \le \sqrt{4\kappa_2} e^{-\kappa t}.$$
	Therefore \eqref{eq:exponential-ergodicity-n} hold.
	
	(b) Finally we argue that $\jointpre^{n,z}(\infty)$ is invariant with respect to $P_t^{n,z}$.
	Noting that $\jointpre^{n,z}(\infty) \in A_n$, we can apply \eqref{eq:exponential-ergodicity-n-claim} and use \eqref{eq:exponential-ergodicity-n-2} to get
	\begin{equation*}
		\limsup_{s \to \infty} W_2(P_t^{n,z}\jointpre^{n,z}(\infty),\jointpre^{n,z}(t+s)) \le \limsup_{s \to \infty} e^{-\kappa t} W_2(\jointpre^{n,z}(\infty),\jointpre^{n,z}(s)) = 0
	\end{equation*}
	and
	\begin{equation*}
%		\label{eq:exponential-ergodicity-n-3}
		\limsup_{s \to \infty} W_2(\jointpre^{n,z}(t+s),\jointpre^{n,z}(s)) \le \limsup_{s \to \infty} e^{-\kappa s} W_2(\jointpre^{n,z}(t),\jointpre^{n,z}(0)) = 0.
	\end{equation*}	
	Combining these two with \eqref{eq:exponential-ergodicity-n-2} gives
	\begin{align*}
		W_2(P_t^{n,z}\jointpre^{n,z}(\infty),\jointpre^{n,z}(\infty))
		& \le \limsup_{s \to \infty} W_2(P_t^{n,z}\jointpre^{n,z}(\infty),\jointpre^{n,z}(t+s)) \\
		& \quad + \limsup_{s \to \infty} W_2(\jointpre^{n,z}(t+s),\jointpre^{n,z}(s)) + \limsup_{s \to \infty} W_2(\jointpre^{n,z}(s),\jointpre^{n,z}(\infty)) \\
		& = 0.
	\end{align*}
	This gives part (b) and completes the proof.
\end{proof}

\subsection{Proofs for Section \ref{sec:LLN}}
\label{sec:pf-LLN}

We need the following lemma to prove Theorem \ref{thm:moment-convergence-uniform}.

\begin{Lemma}
	\label{lem:truncation}
	Suppose Conditions \ref{cond:G-continuous} and \ref{cond:G_n-step-graphon} hold.
	For $s \ge 0$, write
	\begin{align}
		\Rmc^n_s & := \frac{1}{n} \sum_{i=1}^n \Emb \left| \frac{1}{n} \sum_{j=1}^n \int_{\Rmb^d} b(\Xlimit_{\frac{i}{n}}(s),x)G_n(\frac{i}{n},\frac{j}{n})\,\lawlimit_{\frac{j}{n},s}(dx) \right. \notag \\
		& \qquad \left. - \int_I \int_{\Rmb^d} b(\Xlimit_{\frac{i}{n}}(s),x)G(\frac{i}{n},v)\,\lawlimit_{v,s}(dx)\,dv \right|^2. \label{eq:R-n-s}
	\end{align}
	Then 
	\begin{equation*}
		\lim_{n \to \infty} \sup_{s \ge 0} \Rmc^n_s = 0.
	\end{equation*}
\end{Lemma}

\begin{proof}
%[Proof of Lemma \ref{lem:truncation}]
	Fix $M > 1$ and write
	\begin{equation}
		\label{eq:b-M}
		b_M(x,y) := b(x,y) \one_{\{|x| \le M, |y| \le M\}}.
	\end{equation}
	It then follows from \cite[Corollary 2 of Theorem 3.1]{Schultz1969multivariate} that there exist some $m \in \Nmb$ and polynomials
	\begin{equation}
		\label{eq:btil-m}
		\btil_m(x,y) := \sum_{k=1}^m a_k(x)c_k(y) \one_{\{|x| \le M, |y| \le M\}},
	\end{equation}
	where $a_k$ and $c_k$ are polynomials for each $k=1,\dotsc,m$, such that
	\begin{equation}
		\label{eq:b-M-m}
		|b_M(x,y)-\btil_m(x,y)| \le 1/M.
	\end{equation}	
	By adding and subtracting terms, we have
	\begin{align}
		\Rmc^n_s 
		& \le \frac{5}{n} \sum_{i=1}^n \Emb \left| \frac{1}{n} \sum_{j=1}^n \int_{\Rmb^d} \left( b(\Xlimit_{\frac{i}{n}}(s),x) - b_M(\Xlimit_{\frac{i}{n}}(s),x) \right) G_n(\frac{i}{n},\frac{j}{n})\,\lawlimit_{\frac{j}{n},s}(dx) \right|^2 \notag \\
		& \quad + \frac{5}{n} \sum_{i=1}^n \Emb \left| \int_I \int_{\Rmb^d} \left( b(\Xlimit_{\frac{i}{n}}(s),x) - b_M(\Xlimit_{\frac{i}{n}}(s),x) \right) G(\frac{i}{n},v)\,\lawlimit_{v,s}(dx)\,dv \right|^2 \notag \\
		& \quad + \frac{5}{n} \sum_{i=1}^n \Emb \left| \frac{1}{n} \sum_{j=1}^n \int_{\Rmb^d} \left( b_M(\Xlimit_{\frac{i}{n}}(s),x) - \btil_m(\Xlimit_{\frac{i}{n}}(s),x) \right) G_n(\frac{i}{n},\frac{j}{n})\,\lawlimit_{\frac{j}{n},s}(dx) \right|^2 \notag \\
		& \quad + \frac{5}{n} \sum_{i=1}^n \Emb \left| \int_I \int_{\Rmb^d} \left( b_M(\Xlimit_{\frac{i}{n}}(s),x) - \btil_m(\Xlimit_{\frac{i}{n}}(s),x) \right) G(\frac{i}{n},v)\,\lawlimit_{v,s}(dx)\,dv \right|^2 \notag \\
		& \quad + \frac{5}{n} \sum_{i=1}^n \Emb \left| \frac{1}{n} \sum_{j=1}^n \int_{\Rmb^d} \btil_m(\Xlimit_{\frac{i}{n}}(s),x)G_n(\frac{i}{n},\frac{j}{n})\,\lawlimit_{\frac{j}{n},s}(dx) \right. \notag \\
		& \qquad \left. - \int_I \int_{\Rmb^d} \btil_m(\Xlimit_{\frac{i}{n}}(s),x)G(\frac{i}{n},v)\,\lawlimit_{v,s}(dx)\,dv \right|^2 \notag \\
		& =: 5 \sum_{k=1}^5 \Rmc_s^{n,k}. \label{eq:truncation-0}
	\end{align}
	
	Next we analyze each term.
	For $\Rmc_s^{n,1}$ and $\Rmc_s^{n,2}$, using \eqref{eq:b-M}, Remark \ref{rmk:linear-growth}(b), Proposition \ref{prop:finite-moment-uniform} and the Cauchy-Schwarz inequality we have
	\begin{align}
		\Rmc_s^{n,1} 
%		& \le \frac{1}{n} \sum_{i=1}^n \Emb \left[ \frac{1}{n} \sum_{j=1}^n \int_{\Rmb^d} \left( b(\Xlimit_{\frac{i}{n}}(s),x) - b_M(\Xlimit_{\frac{i}{n}}(s),x) \right) G_n(\frac{i}{n},\frac{j}{n})\,\lawlimit_{\frac{j}{n},s}(dx) \right]^2 \notag \\
		& \le \frac{C}{n} \sum_{i=1}^n \Emb \left[ \frac{1}{n} \sum_{j=1}^n \int_{\Rmb^d} \left( 1+|\Xlimit_{\frac{i}{n}}(s)|+|x| \right) \left( \one_{\{|\Xlimit_{\frac{i}{n}}(s)| > M\}} + \one_{\{|x| > M\}} \right) \lawlimit_{\frac{j}{n},s}(dx) \right]^2 \notag \\
		& \le \frac{C}{n} \sum_{i=1}^n \Emb \left[ \left( 1+|\Xlimit_{\frac{i}{n}}(s)| \right)^2 \one_{\{|\Xlimit_{\frac{i}{n}}(s)| > M\}} \right] \notag \\
		& \le \frac{C}{\sqrt{M}}, \label{eq:truncation-1}
	\end{align}
	and
	\begin{align}
		\Rmc_s^{n,2} 
%		\le \frac{1}{n} \sum_{i=1}^n \Emb \left| \int_I \int_{\Rmb^d} \left( b(\Xlimit_{\frac{i}{n}}(s),x) - b_M(\Xlimit_{\frac{i}{n}}(s),x) \right) G(\frac{i}{n},v)\,\lawlimit_{v,s}(dx)\,dv \right|^2 \notag \\
		& \le \frac{C}{n} \sum_{i=1}^n \Emb \left[ \int_I \int_{\Rmb^d} \left( 1+|\Xlimit_{\frac{i}{n}}(s)|+|x| \right) \left( \one_{\{|\Xlimit_{\frac{i}{n}}(s)| > M\}} + \one_{\{|x| > M\}} \right) \lawlimit_{v,s}(dx)\,dv \right]^2 \notag \\
		& \le \frac{C}{n} \sum_{i=1}^n \Emb \left[ \left( 1+|\Xlimit_{\frac{i}{n}}(s)| \right)^2 \one_{\{|\Xlimit_{\frac{i}{n}}(s)| > M\}} \right] + C \int_I \Emb \left[ \left( 1+|\Xlimit_v(s)| \right)^2 \one_{\{|\Xlimit_v(s)| > M\}} \right] dv \notag \\
		& \le \frac{C}{\sqrt{M}}. \label{eq:truncation-2}
	\end{align}	
	For $\Rmc_s^{n,3}$ and $\Rmc_s^{n,4}$, using \eqref{eq:b-M-m} we have
	\begin{equation}
		\label{eq:truncation-3}
		\Rmc_s^{n,3} \le \frac{C}{M^2}, \quad \Rmc_s^{n,4} \le \frac{C}{M^2}.
	\end{equation}
	For $\Rmc_s^{n,5}$, using the step graphon structure \eqref{eq:G_n-step-graphon} of $G_n$ and by adding and subtracting terms, we have
	\begin{align}
		\Rmc_s^{n,5} & = \int_I \Emb \left| \int_I \int_{\Rmb^d} \btil_m(\Xlimit_{\frac{\lceil nu \rceil}{n}}(s),x)G_n(u,v)\,\lawlimit_{\frac{\lceil nv \rceil}{n},s}(dx)\,dv \right. \notag \\
		& \qquad \left. - \int_I \int_{\Rmb^d} \btil_m(\Xlimit_{\frac{\lceil nu \rceil}{n}}(s),x)G(\frac{\lceil nu \rceil}{n},v)\,\lawlimit_{v,s}(dx)\,dv \right|^2 du \notag \\
		& \le 3 \int_I \Emb \left| \int_I \int_{\Rmb^d} \btil_m(\Xlimit_{\frac{\lceil nu \rceil}{n}}(s),x) \left( G_n(u,v) - G(u,v) \right) \lawlimit_{\frac{\lceil nv \rceil}{n},s}(dx)\,dv \right|^2 du \notag \\
		& \quad + 3 \int_I \Emb \left| \int_I \int_{\Rmb^d} \btil_m(\Xlimit_{\frac{\lceil nu \rceil}{n}}(s),x) \left( G(u,v) - G(\frac{\lceil nu \rceil}{n},v) \right) \lawlimit_{\frac{\lceil nv \rceil}{n},s}(dx)\,dv \right|^2 du \notag \\
		& \quad + 3 \int_I \Emb \left| \int_I \int_{\Rmb^d} \btil_m(\Xlimit_{\frac{\lceil nu \rceil}{n}}(s),x)G(\frac{\lceil nu \rceil}{n},v) \left( \lawlimit_{\frac{\lceil nv \rceil}{n},s}(dx) - \lawlimit_{v,s}(dx) \right) dv \right|^2 du \notag \\
		& =: \Rmc_s^{n,6} + \Rmc_s^{n,7} + \Rmc_s^{n,8}. \label{eq:truncation-4}
	\end{align}
	For $\Rmc_s^{n,6}$, using the definition of $\btil_m$ in \eqref{eq:btil-m}, Proposition \ref{prop:finite-moment-uniform} and Remark \ref{rmk:graphon-convergence}, we have
	\begin{align*}
		\Rmc_s^{n,6} & \le 3m \sum_{k=1}^m \int_I \Emb \left[ a_k^2(\Xlimit_{\frac{\lceil nu \rceil}{n}}(s)) \right]  \\
		& \qquad \cdot \left| \int_I \left( G_n(u,v) - G(u,v) \right) \left( \int_{\Rmb^d} c_k(x) \one_{\{|x| \le M\}} \,\lawlimit_{\frac{\lceil nv \rceil}{n},s}(dx) \right) dv \right|^2 du  \\
		& \le C_M \|G_n-G\|, 
%		\label{eq:truncation-5}
	\end{align*}
	where $C_M$ depends on $M$ but not on $n$ or $s$. 
	For $\Rmc_s^{n,7}$, we have
	\begin{equation*}
%		\label{eq:truncation-6}
		\Rmc_s^{n,7} \le C_M \int_{I \times I} \left| G(u,v) - G(\frac{\lceil nu \rceil}{n},v) \right| du\,dv.
	\end{equation*}
	For $\Rmc_s^{n,8}$, using \eqref{eq:Wasserstein-duality} and the Lipschitz property of $b$ (and hence $\btil_m$), we have  
	\begin{equation*}
%		\label{eq:truncation-7}
		\Rmc_s^{n,8} \le C_M \int_I W_2^2 (\lawlimit_{\frac{\lceil nv \rceil}{n},s}, \lawlimit_{v,s}) \,dv.
	\end{equation*}	
	Combining above three estimates with \eqref{eq:truncation-4} and using Remark \ref{rmk:graphon-convergence}, Condition \ref{cond:G_n-step-graphon}, Condition \ref{cond:G-continuous} and Corollary \ref{cor:continuity-Lipschitz-special}(a) gives
	\begin{equation*}
		\lim_{n \to \infty} \sup_{s \ge 0} \Rmc_s^{n,5} = 0.
	\end{equation*}
	Combining this with \eqref{eq:truncation-0}--\eqref{eq:truncation-3} gives
	\begin{equation*}
		\limsup_{n \to \infty} \sup_{s \ge 0} \Rmc^n_s \le \frac{C}{\sqrt{M}}.
	\end{equation*}
	Taking $\limsup_{M \to \infty}$ completes the proof.
\end{proof}

\begin{proof}[Proof of Theorem \ref{thm:moment-convergence-uniform}]
	(a)
	Using It\^o's formula, we have
	\begin{align*}
		& \frac{1}{n} \sum_{i=1}^n \Emb |X_i^n(t)-\Xlimit_{\frac{i}{n}}(t)|^2 \\
		& = \frac{1}{n} \sum_{i=1}^n \Emb \int_0^t 2(X_i^n(s)-\Xlimit_{\frac{i}{n}}(s)) \cdot \left( f(X_i^n(s)) - f(\Xlimit_{\frac{i}{n}}(s)) \right. \\
		& \qquad \left. + \frac{1}{n} \sum_{j=1}^n \xi_{ij}^n b(X_i^n(s),X_j^n(s)) - \int_I \int_{\Rmb^d} b(\Xlimit_{\frac{i}{n}}(s),x)G(\frac{i}{n},v)\,\lawlimit_{v,s}(dx)\,dv \right) ds.
	\end{align*}
	This implies that the function 
	\begin{equation}
		\label{eq:moment-convergence-uniform-0}
		\alpha^n(t) := \frac{1}{n} \sum_{i=1}^n \Emb |X_i^n(t)-\Xlimit_{\frac{i}{n}}(t)|^2
	\end{equation} 
	is differentiable, and
	\begin{align}
		& \frac{1}{n} \sum_{i=1}^n \Emb |X_i^n(t)-\Xlimit_{\frac{i}{n}}(t)|^2 - \frac{1}{n} \sum_{i=1}^n \Emb |X_i^n(r)-\Xlimit_{\frac{i}{n}}(r)|^2 \notag \\
		& = \frac{1}{n} \sum_{i=1}^n \Emb \int_r^t 2(X_i^n(s)-\Xlimit_{\frac{i}{n}}(s)) \cdot \left( f(X_i^n(s)) - f(\Xlimit_{\frac{i}{n}}(s)) \right. \notag \\
		& \qquad \left. + \frac{1}{n} \sum_{j=1}^n \xi_{ij}^n b(X_i^n(s),X_j^n(s)) - \int_I \int_{\Rmb^d} b(\Xlimit_{\frac{i}{n}}(s),x)G(\frac{i}{n},v)\,\lawlimit_{v,s}(dx)\,dv \right) ds. \label{eq:moment-convergence-uniform-1}
	\end{align}
	
	For each $s \ge 0$, using \eqref{eq:c0} we have
	\begin{equation}
		\frac{1}{n} \sum_{i=1}^n \Emb \left[(X_i^n(s)-\Xlimit_{\frac{i}{n}}(s)) \cdot \left( f(X_i^n(s)) - f(\Xlimit_{\frac{i}{n}}(s)) \right) \right] \le -c_0 \frac{1}{n} \sum_{i=1}^n \Emb \left[ |X_i^n(s)-\Xlimit_{\frac{i}{n}}(s)|^2 \right]. \label{eq:moment-convergence-uniform-2}
	\end{equation}
	For the rest in the integrand of \eqref{eq:moment-convergence-uniform-1}, by adding and subtracting terms, we have
	\begin{align}
		& \frac{1}{n} \sum_{i=1}^n \Emb \left[(X_i^n(s)-\Xlimit_{\frac{i}{n}}(s)) \right. \notag \\
		& \qquad \left. \cdot \left( \frac{1}{n} \sum_{j=1}^n \xi_{ij}^n b(X_i^n(s),X_j^n(s)) - \int_I \int_{\Rmb^d} b(\Xlimit_{\frac{i}{n}}(s),x)G(\frac{i}{n},v)\,\lawlimit_{v,s}(dx)\,dv \right) \right] \notag \\
		& = \frac{1}{n} \sum_{i=1}^n \Emb \left[(X_i^n(s)-\Xlimit_{\frac{i}{n}}(s)) \cdot \left( \frac{1}{n} \sum_{j=1}^n \xi_{ij}^n \left( b(X_i^n(s),X_j^n(s)) - b(\Xlimit_{\frac{i}{n}}(s),\Xlimit_{\frac{j}{n}}(s)) \right) \right) \right] \notag \\	
		& \quad + \frac{1}{n} \sum_{i=1}^n \Emb \left[(X_i^n(s)-\Xlimit_{\frac{i}{n}}(s)) \right. \notag \\
		& \qquad \left. \cdot \left( \frac{1}{n} \sum_{j=1}^n \left( \xi_{ij}^n b(\Xlimit_{\frac{i}{n}}(s),\Xlimit_{\frac{j}{n}}(s)) - \int_{\Rmb^d} b(\Xlimit_{\frac{i}{n}}(s),x)G_n(\frac{i}{n},\frac{j}{n})\,\lawlimit_{\frac{j}{n},s}(dx) \right) \right) \right] \notag \\
		& \quad + \frac{1}{n} \sum_{i=1}^n \Emb \left[(X_i^n(s)-\Xlimit_{\frac{i}{n}}(s)) \right. \notag \\
		& \qquad \left. \cdot \left( \frac{1}{n} \sum_{j=1}^n \int_{\Rmb^d} b(\Xlimit_{\frac{i}{n}}(s),x)G_n(\frac{i}{n},\frac{j}{n})\,\lawlimit_{\frac{j}{n},s}(dx) - \int_I \int_{\Rmb^d} b(\Xlimit_{\frac{i}{n}}(s),x)G(\frac{i}{n},v)\,\lawlimit_{v,s}(dx)\,dv \right) \right] \notag \\
		& =: \Smc_s^{n,1} + \Smc_s^{n,2} + \Smc_s^{n,3}. \label{eq:moment-convergence-uniform-3}
%		& \le  + C\left( \frac{1}{n} \sum_{i=1}^n \Emb |X_i^n(s)-\Xlimit_{\frac{i}{n}}(s)|^2 \right)^{1/2} \left( \Tmc_s^{n,2} + \Tmc_s^{n,3} \right)^{1/2},
	\end{align}
	For $\Smc_s^{n,1}$, using the Lipschitz property of $b$ and Young's inequality, we have
	\begin{equation}
		\Smc_s^{n,1} \le 2K_b\frac{1}{n} \sum_{i=1}^n \Emb |X_i^n(s)-\Xlimit_{\frac{i}{n}}(s)|^2. \label{eq:moment-convergence-uniform-4}
	\end{equation}
	For $\Smc_s^{n,2}$, using the Cauchy-Schwarz inequality we have
	\begin{align*}
		\Smc_s^{n,2} & \le \left( \frac{1}{n} \sum_{i=1}^n \Emb |X_i^n(s)-\Xlimit_{\frac{i}{n}}(s)|^2 \right)^{1/2} \\
		& \quad \cdot \left( \frac{1}{n} \sum_{i=1}^n \Emb \left| \frac{1}{n} \sum_{j=1}^n \left( \xi_{ij}^n b(\Xlimit_{\frac{i}{n}}(s),\Xlimit_{\frac{j}{n}}(s)) - \int_{\Rmb^d} b(\Xlimit_{\frac{i}{n}}(s),x)G_n(\frac{i}{n},\frac{j}{n})\,\lawlimit_{\frac{j}{n},s}(dx) \right) \right|^2 \right)^{1/2}. 
	\end{align*}
	Due to the independence of $\xi_{ij}^n$ and $\Xlimit_u$, we have
	\begin{align}
		& \Emb \left| \frac{1}{n} \sum_{j=1}^n \left( \xi_{ij}^n b(\Xlimit_{\frac{i}{n}}(s),\Xlimit_{\frac{j}{n}}(s)) - \int_{\Rmb^d} b(\Xlimit_{\frac{i}{n}}(s),x)G_n(\frac{i}{n},\frac{j}{n})\,\lawlimit_{\frac{j}{n},s}(dx) \right) \right|^2 \notag \\
		& = \frac{1}{n^2} \sum_{j=1}^n \Emb \left| \xi_{ij}^n b(\Xlimit_{\frac{i}{n}}(s),\Xlimit_{\frac{j}{n}}(s)) - \int_{\Rmb^d} b(\Xlimit_{\frac{i}{n}}(s),x)G_n(\frac{i}{n},\frac{j}{n})\,\lawlimit_{\frac{j}{n},s}(dx) \right|^2 \notag \\
		& \le \frac{C}{n}, \label{eq:moment-convergence-uniform-9}
	\end{align}	
	where the last line uses Remark \ref{rmk:linear-growth}(b) and Proposition \ref{prop:finite-moment-uniform}.
	Therefore
	\begin{equation}
		\Smc_s^{n,2} \le \frac{C}{\sqrt{n}} \left( \frac{1}{n} \sum_{i=1}^n \Emb |X_i^n(s)-\Xlimit_{\frac{i}{n}}(s)|^2 \right)^{1/2}. \label{eq:moment-convergence-uniform-5}
	\end{equation}
	For $\Smc_s^{n,3}$, using the Cauchy-Schwarz inequality and the definition of $\Rmc_s^n$ in \eqref{eq:R-n-s}, we have 
	\begin{equation}
		\Smc_s^{n,3} \le (\Rmc_s^n)^{1/2} \left( \frac{1}{n} \sum_{i=1}^n \Emb |X_i^n(s)-\Xlimit_{\frac{i}{n}}(s)|^2 \right)^{1/2}. \label{eq:moment-convergence-uniform-6}
	\end{equation}
	
	Combining \eqref{eq:moment-convergence-uniform-0}--\eqref{eq:moment-convergence-uniform-4}, \eqref{eq:moment-convergence-uniform-5} and \eqref{eq:moment-convergence-uniform-6}, we have
	\begin{align*}
		\alpha^n(t) - \alpha^n(r) 
		& \le -2(c_0-2K_b) \int_r^t \alpha^n(s) \, ds 
		+ \left(\frac{C}{\sqrt{n}} + 2\sup_{s \ge 0} \sqrt{\Rmc_s^n} \right) \int_r^t \sqrt{\alpha^n(s)} \,ds.
	\end{align*}
	Recall that the function $\alpha^n(t)$ is differentiable, non-negative, and $\alpha^n(0)=0$.	
	It then follows from Lemma \ref{lem:ODE} (with $a_1=2(c_0-2K_b)$, $a_2=\frac{C}{\sqrt{n}} + 2\sup_{s \ge 0} \sqrt{\Rmc_s^n}$, $a_3=0$) that
%	\begin{equation*}
%		\alpha^n'(t) \le -2(c_0-2K_b) \alpha^n(t) + \left( \frac{C}{\sqrt{n}} + \sup_{s \ge 0} (\Rmc_s^n)^{1/2} \right) \sqrt{\alpha^n(t)}.
%	\end{equation*}
%	Noting that the right hand side is negative when $\sqrt{\alpha^n(t)} > \left( \frac{C}{\sqrt{n}} + \sup_{s \ge 0} (\Rmc_s^n)^{1/2} \right) / 2(c_0-2K_b)$, we must have
	\begin{equation*}
%		\label{eq:claim}
		\alpha^n(t) \le C \left( \frac{1}{n} + \sup_{s \ge 0} \Rmc_s^n \right).
	\end{equation*}
	Combining this with Lemma \ref{lem:truncation} gives part (a).
	
	(b) Next we prove the first convergence statement in \eqref{eq:LLN-uniform}.
	Write
	\begin{equation}
		\label{eq:lawlimit-n}
		\lawlimit^n(t) := \frac{1}{n} \sum_{i=1}^n \lawlimit_{\frac{i}{n},t}.
	\end{equation}
	Using the triangle inequality we have
	\begin{equation}
		\label{eq:LLN-uniform-pf-1}
		W_2(\lawpre^n(t),\lawlimitmean(t)) \le W_2(\lawpre^n(t),\lawlimit^n(t)) + W_2(\lawlimit^n(t),\lawlimitmean(t)).
	\end{equation}	
	Taking $\pi = \frac{1}{n} \sum_{i=1}^n \Lmc(X_i^n(t),\Xlimit_{\frac{i}{n}}(t))$ as the coupling of $\lawpre^n(t)$ and $\lawlimit^n(t)$ and using part (a), we have
	\begin{equation*}
%		\label{eq:LLN-uniform-pf-2}
		\sup_{t \ge 0} W_2(\lawpre^n(t),\lawlimit^n(t)) \le \sup_{t \ge 0} \left( \frac{1}{n} \sum_{i=1}^n \Emb |X_i^n(t)-\Xlimit_{\frac{i}{n}}(t)|^2 \right)^{1/2} \to 0
	\end{equation*}	
	as $n \to \infty$.
	Using the convexity of $W^2_2(\cdot,\cdot)$ and Corollary \ref{cor:continuity-Lipschitz-special}(a) we have
	\begin{equation}
		\label{eq:LLN-uniform-pf-3}
		\sup_{t \ge 0} W_2^2(\lawlimit^n(t),\lawlimitmean(t)) \le \int_0^1 \sup_{t \ge 0} W_2^2(\lawlimit_{\frac{\lceil nu \rceil}{n},t}, \lawlimit_{u,t}) \,du \to 0
	\end{equation}
	as $n \to \infty$.
	Combining these three displays gives the first convergence in part (b).
	
	Finally, for the second convergence statement in part (b), let
	\begin{equation}
		\label{eq:emplimit-n}
		\emplimit^n(t) := \frac{1}{n} \sum_{i=1}^n \delta_{\Xlimit_{\frac{i}{n}}(t)}.
	\end{equation}
	From the triangle inequality we have
	\begin{align}
		\Emb W_2(\emppre^n(t),\lawlimitmean(t)) & \le \Emb W_2(\emppre^n(t),\emplimit^n(t)) + \Emb W_2(\emplimit^n(t),\lawlimit^n(t)) + W_2(\lawlimit^n(t),\lawlimitmean(t)). \label{eq:LLN-uniform-pf-5}
	\end{align}
	In view of \eqref{eq:LLN-uniform-pf-3}, it suffices to show
	\begin{equation}
		\label{eq:LLN-uniform-claim}
		\sup_{t \ge 0} \Emb W_2(\emppre^n(t),\emplimit^n(t)) + \sup_{t \ge 0} \Emb W_2(\emplimit^n(t),\lawlimit^n(t)) \to 0
	\end{equation}
	as $n \to \infty$.
	Taking $\pi = \frac{1}{n} \sum_{i=1}^n \delta_{(X_i^n(t),\Xlimit_{\frac{i}{n}}(t))}$ as the coupling of $\emppre^n(t)$ and $\emplimit^n(t)$ and using part (a), we have
	\begin{equation*}
		\sup_{t \ge 0} \Emb W_2(\emppre^n(t),\emplimit^n(t)) \le \sup_{t \ge 0} \left( \frac{1}{n} \sum_{i=1}^n \Emb |X_i^n(t)-\Xlimit_{\frac{i}{n}}(t)|^2 \right)^{1/2} \to 0
	\end{equation*}
	as $n \to \infty$.
	Applying Lemma \ref{lem:Wasserstein-empirical} with $Y_i = \Xlimit_{\frac{i}{n}}$, $p=3$ and $q=4$, we have
	\begin{equation*}
		\Emb W_2(\emplimit^n(t),\lawlimit^n(t)) \le \left( \Emb W_p^p(\emplimit^n(t),\lawlimit^n(t)) \right)^{1/p} \le C \left( \int_{\Rmb^d} |x|^q \, \lawlimit^n(t)(dx) \right)^{1/q} a(n),
	\end{equation*}
	where $a(n) = n^{-1/d}+n^{-1/12}$ is defined in \eqref{eq:an}.
	It then follows from Proposition \ref{prop:finite-moment-uniform} that
	\begin{equation}
		\label{eq:LLN-uniform-pf-6}
		\sup_{t \ge 0} \Emb W_2(\emplimit^n(t),\lawlimit^n(t)) \le C a(n) \to 0
	\end{equation}
	as $n \to \infty$.
	Therefore \eqref{eq:LLN-uniform-claim} holds and hence the second convergence in part (b) holds.		
	This completes the proof.
\end{proof}

\begin{Remark}
	\label{rmk:choice-of-p}
	The choice of $p=3$ above \eqref{eq:LLN-uniform-pf-6} could be replaced by any $2 < p < 4$.
	As a result, the constant $C$ and rate $a(n)$ will change accordingly, by Lemma \ref{lem:Wasserstein-empirical}.
\end{Remark}

\begin{proof}[Proof of Theorem \ref{thm:moment-convergence-uniform-Lipschitz}]
	(a) 
	Similar to the proof of Theorem \ref{thm:moment-convergence-uniform}(a), we apply It\^{o}'s formula and get
	\begin{align*}
		\Emb |X_i^n(t)-\Xlimit_{\frac{i}{n}}(t)|^2
		& = \Emb \int_0^t 2(X_i^n(s)-\Xlimit_{\frac{i}{n}}(s)) \cdot \left( f(X_i^n(s)) - f(\Xlimit_{\frac{i}{n}}(s)) \right. \\
		& \qquad \left. + \frac{1}{n} \sum_{j=1}^n \xi_{ij}^n b(X_i^n(s),X_j^n(s)) - \int_I \int_{\Rmb^d} b(\Xlimit_{\frac{i}{n}}(s),x)G(\frac{i}{n},v)\,\lawlimit_{v,s}(dx)\,dv \right) ds.
	\end{align*}
	This implies that the functions 
	\begin{equation}
		\label{eq:moment-convergence-uniform-Lipschitz-1}
		\bar{\alpha}_i^n(t) := \Emb |X_i^n(t)-\Xlimit_{\frac{i}{n}}(t)|^2 \quad \text{ and } \quad \bar{\alpha}^n(t) := \frac{1}{n} \sum_{i=1}^n \Emb |X_i^n(t)-\Xlimit_{\frac{i}{n}}(t)|^2
	\end{equation} 
	are differentiable, and
	\begin{align}
		& \Emb |X_i^n(t)-\Xlimit_{\frac{i}{n}}(t)|^2 - \Emb |X_i^n(r)-\Xlimit_{\frac{i}{n}}(r)|^2 \notag \\
		& = \Emb \int_r^t 2(X_i^n(s)-\Xlimit_{\frac{i}{n}}(s)) \cdot \left( f(X_i^n(s)) - f(\Xlimit_{\frac{i}{n}}(s)) \right. \notag \\
		& \qquad \left. + \frac{1}{n} \sum_{j=1}^n \xi_{ij}^n b(X_i^n(s),X_j^n(s)) - \int_I \int_{\Rmb^d} b(\Xlimit_{\frac{i}{n}}(s),x)G(\frac{i}{n},v)\,\lawlimit_{v,s}(dx)\,dv \right) ds. \label{eq:moment-convergence-uniform-Lipschitz-2}
	\end{align}
	
	For each $s \ge 0$, using \eqref{eq:c0} we have
	\begin{equation}
		\label{eq:moment-convergence-uniform-Lipschitz-3}
		\Emb \left[(X_i^n(s)-\Xlimit_{\frac{i}{n}}(s)) \cdot \left( f(X_i^n(s)) - f(\Xlimit_{\frac{i}{n}}(s)) \right) \right] \le -c_0 \Emb |X_i^n(s)-\Xlimit_{\frac{i}{n}}(s)|^2.
	\end{equation}
	For the rest in the integrand, by adding and subtracting terms and using Condition \ref{cond:G_n=G}, we have
	\begin{align}
		& \Emb \left[(X_i^n(s)-\Xlimit_{\frac{i}{n}}(s)) \right. \notag \\
		& \qquad \left. \cdot \left( \frac{1}{n} \sum_{j=1}^n \xi_{ij}^n b(X_i^n(s),X_j^n(s)) - \int_I \int_{\Rmb^d} b(\Xlimit_{\frac{i}{n}}(s),x)G(\frac{i}{n},v)\,\lawlimit_{v,s}(dx)\,dv \right) \right] \notag \\
		& = \Emb \left[(X_i^n(s)-\Xlimit_{\frac{i}{n}}(s)) \cdot \left( \frac{1}{n} \sum_{j=1}^n \xi_{ij}^n \left( b(X_i^n(s),X_j^n(s)) - b(\Xlimit_{\frac{i}{n}}(s),\Xlimit_{\frac{j}{n}}(s)) \right) \right) \right] \notag \\	
		& \quad + \Emb \left[(X_i^n(s)-\Xlimit_{\frac{i}{n}}(s)) \right. \notag \\
		& \qquad \left. \cdot \left( \frac{1}{n} \sum_{j=1}^n \left( \xi_{ij}^n b(\Xlimit_{\frac{i}{n}}(s),\Xlimit_{\frac{j}{n}}(s)) - \int_{\Rmb^d} b(\Xlimit_{\frac{i}{n}}(s),x)G(\frac{i}{n},\frac{j}{n})\,\lawlimit_{\frac{j}{n},s}(dx) \right) \right) \right] \notag \\
		& \quad + \Emb \left[(X_i^n(s)-\Xlimit_{\frac{i}{n}}(s)) \right. \notag \\
		& \qquad \left. \cdot \left( \frac{1}{n} \sum_{j=1}^n \int_{\Rmb^d} b(\Xlimit_{\frac{i}{n}}(s),x)G(\frac{i}{n},\frac{j}{n})\,\lawlimit_{\frac{j}{n},s}(dx) - \int_I \int_{\Rmb^d} b(\Xlimit_{\frac{i}{n}}(s),x)G(\frac{i}{n},v)\,\lawlimit_{v,s}(dx)\,dv \right) \right] \notag \\
		& =: \Smcbar_s^{n,i,1} + \Smcbar_s^{n,i,2} + \Smcbar_s^{n,i,3}. \label{eq:moment-convergence-uniform-Lipschitz-4}
%		& \le  + C\left( \Emb |X_i^n(s)-\Xlimit_{\frac{i}{n}}(s)|^2 \right)^{1/2} \left( \Smcbar_s^{n,2} + \Smcbar_s^{n,3} \right)^{1/2},
	\end{align}	
	For $\Smcbar_s^{n,i,1}$, using the Cauchy-Schwarz inequality, the Lipschitz property of $b$ and Young's inequality, we have
	\begin{align}
		\Smcbar_s^{n,i,1} & \le \Emb \left[ \left|X_i^n(s)-\Xlimit_{\frac{i}{n}}(s)\right| \cdot \frac{K_b}{n} \sum_{j=1}^n \left( \left| X_i^n(s) - \Xlimit_{\frac{i}{n}}(s) \right| + \left|X_j^n(s)-\Xlimit_{\frac{j}{n}}(s) \right| \right) \right] \notag \\
		& \le \frac{3K_b}{2} \Emb |X_i^n(s)-\Xlimit_{\frac{i}{n}}(s)|^2 + \frac{K_b}{2} \frac{1}{n} \sum_{j=1}^n \Emb |X_j^n(s)-\Xlimit_{\frac{j}{n}}(s)|^2. \label{eq:moment-convergence-uniform-Lipschitz-5}
	\end{align}
	For $\Smcbar_s^{n,i,2}$, using the Cauchy-Schwarz inequality and the weak LLN type estimate \eqref{eq:moment-convergence-uniform-9}, we have
	\begin{align}
		\Smcbar_s^{n,i,2} & \le \left( \Emb |X_i^n(s)-\Xlimit_{\frac{i}{n}}(s)|^2 \right)^{1/2} \notag \\
		& \quad \cdot \left( \Emb \left| \frac{1}{n} \sum_{j=1}^n \left( \xi_{ij}^n b(\Xlimit_{\frac{i}{n}}(s),\Xlimit_{\frac{j}{n}}(s)) - \int_{\Rmb^d} b(\Xlimit_{\frac{i}{n}}(s),x)G_n(\frac{i}{n},\frac{j}{n})\,\lawlimit_{\frac{j}{n},s}(dx) \right) \right|^2 \right)^{1/2} \notag \\
		& \le \frac{C}{\sqrt{n}} \left( \Emb |X_i^n(s)-\Xlimit_{\frac{i}{n}}(s)|^2 \right)^{1/2}. \label{eq:moment-convergence-uniform-Lipschitz-6}
	\end{align}
	For $\Smcbar_s^{n,i,3}$, note that
	\begin{align*}
		& \Emb \left| \frac{1}{n} \sum_{j=1}^n \int_{\Rmb^d} b(\Xlimit_{\frac{i}{n}}(s),x)G(\frac{i}{n},\frac{j}{n})\,\lawlimit_{\frac{j}{n},s}(dx) - \int_I \int_{\Rmb^d} b(\Xlimit_{\frac{i}{n}}(s),x)G(\frac{i}{n},v)\,\lawlimit_{v,s}(dx)\,dv \right|^2 \\
		& = \Emb \left| \int_I \int_{\Rmb^d} b(\Xlimit_{\frac{i}{n}}(s),x)G(\frac{i}{n},\frac{\lceil nv \rceil}{n})\,\lawlimit_{\frac{\lceil nv \rceil}{n},s}(dx)\,dv - \int_I \int_{\Rmb^d} b(\Xlimit_{\frac{i}{n}}(s),x)G(\frac{i}{n},v)\,\lawlimit_{v,s}(dx)\,dv \right|^2 \\
		& \le 2 \Emb \left| \int_I \int_{\Rmb^d} b(\Xlimit_{\frac{i}{n}}(s),x) \left( G(\frac{i}{n},\frac{\lceil nv \rceil}{n}) - G(\frac{i}{n},v) \right) \lawlimit_{\frac{\lceil nv \rceil}{n},s}(dx)\,dv \right|^2 \\
		& \quad + 2 \Emb \left| \int_I \int_{\Rmb^d} b(\Xlimit_{\frac{i}{n}}(s),x) \left( \lawlimit_{\frac{\lceil nv \rceil}{n},s}(dx) - \lawlimit_{v,s}(dx) \right) G(\frac{i}{n},v) dv \right|^2 \\
		& \le \frac{C}{n^2},
	\end{align*}
	where the last inequality uses Condition \ref{cond:G-Lipschitz}, Proposition \ref{prop:finite-moment-uniform} and Remark \ref{rmk:linear-growth}(b) for the first term, and the Lipschitz property of $b$, \eqref{eq:Wasserstein-duality} and Corollary \ref{cor:continuity-Lipschitz-special}(b) for the second term.
	Therefore 
	\begin{align}
		\Smcbar_s^{n,i,3} & \le \left( \Emb |X_i^n(s)-\Xlimit_{\frac{i}{n}}(s)|^2 \right)^{1/2} \notag \\
		& \quad \cdot \left( \Emb \left| \frac{1}{n} \sum_{j=1}^n \int_{\Rmb^d} b(\Xlimit_{\frac{i}{n}}(s),x)G(\frac{i}{n},\frac{j}{n})\,\lawlimit_{\frac{j}{n},s}(dx) - \int_I \int_{\Rmb^d} b(\Xlimit_{\frac{i}{n}}(s),x)G(\frac{i}{n},v)\,\lawlimit_{v,s}(dx)\,dv \right|^2 \right)^{1/2} \notag \\
		& \le \frac{C}{n} \left( \Emb |X_i^n(s)-\Xlimit_{\frac{i}{n}}(s)|^2 \right)^{1/2}. \label{eq:moment-convergence-uniform-Lipschitz-7}
	\end{align}
	
	Combining \eqref{eq:moment-convergence-uniform-Lipschitz-1}--\eqref{eq:moment-convergence-uniform-Lipschitz-7}, we have
	\begin{align}		
		\bar{\alpha}_i^n(t) - \bar{\alpha}_i^n(r) &
		\le -(2c_0-3K_b) \int_r^t \bar{\alpha}_i^n(s) \, ds + K_b \int_r^t \bar{\alpha}^n(s) \,ds
		+ \frac{C}{\sqrt{n}} \int_r^t \sqrt{\bar{\alpha}_i^n(s)} \,ds, \label{eq:moment-convergence-uniform-Lipschitz-9}
%		& \le -(c_0-K_b) \int_r^t \bar{\alpha}_i^n(s) \, ds + K_b \int_r^t \bar{\alpha}^n(s) \,ds + \frac{C}{n} (t-r), \label{eq:moment-convergence-uniform-Lipschitz-9}
	\end{align}
%	where the last inequality uses the estimate
%	\begin{equation*}
%		\sqrt{\frac{\bar{\alpha}^n(s)}{n}} \le (c_0-2K_b)\bar{\alpha}^n(s) + \frac{1}{4(c_0-2K_b)n}
%	\end{equation*}
	Taking the average over $i=1,\dotsc,n$ gives
	\begin{equation*}
		\bar{\alpha}^n(t) - \bar{\alpha}^n(r) 
		\le -2(c_0-2K_b) \int_r^t \bar{\alpha}^n(s) \, ds + \frac{C}{\sqrt{n}} \int_r^t \sqrt{\bar{\alpha}^n(s)} \,ds.
	\end{equation*}
	Since the function $\bar{\alpha}^n(t)$ is non-negative and differentiable with $\bar{\alpha}^n(0)=0$, using Lemma \ref{lem:ODE} (with $a_1=2(c_0-2K_b)$, $a_2=\frac{C}{\sqrt{n}}$, $a_3=0$) we have
%	\begin{equation*}
%		\bar{\alpha}^n'(t) \le -(c_0-2K_b) \bar{\alpha}^n(t) + \frac{C}{n}.
%	\end{equation*}
%	Noting that $\bar{\alpha}^n(0)=0$ and the right hand side is negative when $\bar{\alpha}^n(t) > \frac{C}{(c_0-2K_b)n}$, we must have
%	\begin{equation*}
		$\bar{\alpha}^n(t) \le \frac{C}{n}$.
%	\end{equation*}
	From this and \eqref{eq:moment-convergence-uniform-Lipschitz-9} we further have
	\begin{equation*}
		\bar{\alpha}_i^n(t) - \bar{\alpha}_i^n(r) 
		\le -(2c_0-3K_b) \int_r^t \bar{\alpha}_i^n(s) \, ds + \frac{C}{\sqrt{n}} \int_r^t \sqrt{\bar{\alpha}_i^n(s)} \,ds + \frac{C}{n} (t-r).
	\end{equation*}	
	Since the function $\bar{\alpha}_i^n(t)$ is non-negative and differentiable with $\bar{\alpha}_i^n(0)=0$, it follows from Lemma \ref{lem:ODE} again that $\bar{\alpha}_i^n(t) \le \frac{C}{n}$, uniformly in $t \ge 0$, $n \in \Nmb$ and $i=1,\dotsc,n$. 
	This gives part (a).
	
	(b)
	The proof is similar to that of Theorem \ref{thm:moment-convergence-uniform}(b), but we will have better estimates under Conditions \ref{cond:G-Lipschitz} and \ref{cond:G_n=G}.
	Recall $\lawlimit^n(t)$ in \eqref{eq:lawlimit-n}.
	Taking $\pi = \frac{1}{n} \sum_{i=1}^n \Lmc(X_i^n(t),\Xlimit_{\frac{i}{n}}(t))$ as the coupling of $\lawpre^n(t)$ and $\lawlimit^n(t)$ and using part (a), we have
	\begin{equation*}
		\sup_{t \ge 0} W_2(\lawpre^n(t),\lawlimit^n(t)) \le \sup_{t \ge 0} \left( \frac{1}{n} \sum_{i=1}^n \Emb |X_i^n(t)-\Xlimit_{\frac{i}{n}}(t)|^2 \right)^{1/2} \le \frac{C}{\sqrt{n}}.
	\end{equation*}
	Using the convexity of $W^2_2(\cdot,\cdot)$ and Corollary \ref{cor:continuity-Lipschitz-special}(b) we have
	\begin{equation}
		\label{eq:LLN-uniform-Lipschitz-pf-1}
		\sup_{t \ge 0} W_2^2(\lawlimit^n(t),\lawlimitmean(t)) \le \int_0^1 \sup_{t \ge 0} W_2^2(\lawlimit_{\frac{\lceil nu \rceil}{n},t}, \lawlimit_{u,t}) \,du  \le \frac{C}{n^2}.
	\end{equation}
	Combining these two estimates with \eqref{eq:LLN-uniform-pf-1} gives the first statement in part (b).
	
	For the second statement in part (b), recall $\emplimit^n(t)$ in \eqref{eq:emplimit-n}.
	Taking $\pi = \frac{1}{n} \sum_{i=1}^n \delta_{(X_i^n(t),\Xlimit_{\frac{i}{n}}(t))}$ as the coupling of $\emppre^n(t)$ and $\emplimit^n(t)$ and using part (a), we have
	\begin{equation*}
		\sup_{t \ge 0} \Emb W_2(\emppre^n(t),\emplimit^n(t)) \le \sup_{t \ge 0} \left( \frac{1}{n} \sum_{i=1}^n \Emb |X_i^n(t)-\Xlimit_{\frac{i}{n}}(t)|^2 \right)^{1/2} \le \frac{C}{\sqrt{n}}.
	\end{equation*}
	Combining this with \eqref{eq:LLN-uniform-pf-5}, \eqref{eq:LLN-uniform-pf-6} and \eqref{eq:LLN-uniform-Lipschitz-pf-1}
%	Lastly, applying Lemma \ref{lem:Wasserstein-empirical} with $Y_i = \Xlimit_{\frac{i}{n}}$, $p=\sqrt{5}$, $q=4 \notin \{2p,d/(d-p)\}$, we have
%	\begin{equation*}
%		\Emb W_2(\emplimit^n(t),\lawlimit^n(t)) \le \left( \Emb W_p^p(\emplimit^n(t),\lawlimit^n(t)) \right)^{1/p} \le C \left( \int_{\Rmb^d} |x|^q \, \lawlimit^n(t)(dx) \right)^{1/q} a(n),
%	\end{equation*}
%	for some sequence $a(n) \to 0$ as $n \to \infty$.
%	It then follows from Proposition \ref{prop:finite-moment-uniform} that
%	\begin{equation*}
%		\sup_{t \ge 0} \Emb W_2(\emplimit^n(t),\lawlimit^n(t)) \le C a(n) \to 0
%	\end{equation*}
%	as $n \to \infty$.
	gives the second statement in part (b).
	
	(c)
	Finally, for all $n,k \in \Nmb$ and any distinct $i_1,\dotsc,i_k \in \{1,\dotsc,n\}$, taking $\pi = \Lmc\left(\left(X_{i_1}^n(t), \dotsc, X_{i_k}^n(t)\right), \left(\Xlimit_{\frac{i_1}{n}}(t), \dotsc, \Xlimit_{\frac{i_k}{n}}(t)\right)\right)$ as the coupling and using part (a), we have
	\begin{align*}
		\sup_{t \ge 0} W_2(\Lmc(X_{i_1}^n(t), \dotsc, X_{i_k}^n(t)), \, \lawlimit_{\frac{i_1}{n},t} \otimes \dotsb \otimes \lawlimit_{\frac{i_k}{n},t})
		& \le \left( \sum_{j=1}^k \Emb |X_{i_j}^n(t) - \Xlimit_{\frac{i_j}{n}}(t)|^2 \right)^{1/2} 
		\le \frac{C\sqrt{k}}{\sqrt{n}}.
	\end{align*}			
	This gives part (c) and completes the proof.	
\end{proof}

\subsection{Proofs for Section \ref{sec:Euler}}
\label{sec:pf-Euler}

We first show the following uniform-in-time estimates.

\begin{Lemma}
	\label{lem:finite-moment-uniform-Euler}
	There exist $h_0, C \in (0,\infty)$ such that
	\begin{equation*}
		\sup_{n \ge 1} \max_{i=1,\dotsc,n} \Emb |X_i^{n,h}(s)-X_i^{n,h}(s_h)|^2 \le C(s-s_h) \le Ch, \quad \forall \, s \ge 0, \: h \in (0,h_0),
	\end{equation*}
	and 
	\begin{equation*}
		\sup_{h \in (0,h_0)} \sup_{n \ge 1} \max_{i=1,\dotsc,n} \sup_{t \ge 0} \Emb |X_i^{n,h}(t)|^2 \le C.
	\end{equation*}
\end{Lemma}

\begin{proof}
%[Proof of Lemma \ref{lem:finite-moment-uniform-Euler}]
	Before analyzing the system \eqref{eq:system-n-h}, consider the following equivalent discrete-time model: $Z_i^{n,h}(0) = X_i^{n,h}(0)$ and
	\begin{align*}
		Z_i^{n,h}(k+1) & = Z_i^{n,h}(k) + \left( f(Z_i^{n,h}(k)) + \frac{1}{n} \sum_{j=1}^n \xi_{ij}^n b(Z_i^{n,h}(k),Z_j^{n,h}(k)) \right) h + \Delta_k B_{\frac{i}{n}}, \quad k \in \Nmb_0,
	\end{align*} 
	where $\Delta_k B_{\frac{i}{n}} := B_{\frac{i}{n}}((k+1)h) - B_{\frac{i}{n}}(kh)$.
	Note that $Z_i^{n,h}(k) := X_i^{n,h}(kh)$.
	
	We claim that
	\begin{equation}
		\label{eq:finite-moment-uniform-Euler-1}
		\sup_{h \in (0,h_0)} \sup_{n \ge 1} \max_{i=1,\dotsc,n} \sup_{k \in \Nmb_0} \Emb |Z_i^{n,h}(k)|^2 < \infty,
	\end{equation}
	for some $h_0 \in (0,\infty)$.
	To see this, write
	\begin{align*}
		& |Z_i^{n,h}(k+1)|^2 - |Z_i^{n,h}(k)|^2 \\
		& = 2 Z_i^{n,h}(k) \cdot \left( Z_i^{n,h}(k+1) - Z_i^{n,h}(k) \right) + |Z_i^{n,h}(k+1) - Z_i^{n,h}(k)|^2\\
		& =  2 Z_i^{n,h}(k) \cdot f(Z_i^{n,h}(k)) h + 2 Z_i^{n,h}(k) \cdot \left( \frac{1}{n} \sum_{j=1}^n \xi_{ij}^n b(Z_i^{n,h}(k),Z_j^{n,h}(k)) \right) h + \zeta^{n,h}_i(k) \cdot \sigma \Delta_k B_{\frac{i}{n}} \\
		& \quad + \left|f(Z_i^{n,h}(k)) + \frac{1}{n} \sum_{j=1}^n \xi_{ij}^n b(Z_i^{n,h}(k),Z_j^{n,h}(k))\right|^2 h^2 + \left|\sigma \Delta_k B_{\frac{i}{n}}\right|^2, 
	\end{align*}
	where $\zeta^{n,h}_i(k)$ is measurable with respect to $\sigma\{B_{\frac{j}{n}}(s) : 0 \le s \le kh\}$.
	Let $\beta_i^{n,h}(k) := \Emb |Z_i^{n,h}(k)|^2$ and $\beta^{n,h}(k) := \frac{1}{n} \sum_{j=1}^n \Emb|Z_j^{n,h}(k)|^2$.	
	Using \eqref{eq:c0}, the Lipschitz property of $f,b$ and the Cauchy-Schwarz inequality we have
	\begin{align*}
		\beta_i^{n,h}(k+1) - \beta_i^{n,h}(k) 
		& = \Emb |Z_i^{n,h}(k+1)|^2 - \Emb |Z_i^{n,h}(k)|^2 \\
		& \le \left( -2c_0 \beta_i^{n,h}(k) + C\sqrt{\beta_i^{n,h}(k)} + 2K_b \beta_i^{n,h}(k) + 2 K_b \sqrt{\beta_i^{n,h}(k) \beta^{n,h}(k)} \right) h \\
		& \qquad + C \left( 1 + \beta_i^{n,h}(k) + \beta^{n,h}(k) \right) h^2 + Ch.
	\end{align*}
	Since $C\sqrt{\beta_i^{n,h}(k)} \le (c_0-2K_b) \beta_i^{n,h}(k) + \frac{C^2}{4(c_0-2K_b)}$ and $2\sqrt{\beta_i^{n,h}(k) \beta^{n,h}(k)} \le \beta_i^{n,h}(k) + \beta^{n,h}(k)$, we have
	\begin{align}
		& \beta_i^{n,h}(k+1) - \beta_i^{n,h}(k) \label{eq:finite-moment-uniform-Euler-2} \\
		& \le \left( -(c_0-K_b) \beta_i^{n,h}(k) + K_b \beta^{n,h}(k) + C \right) h + C \left( 1 + \beta_i^{n,h}(k) + \beta^{n,h}(k) \right) h^2. \notag
	\end{align}
	Taking the average over $i=1,\dotsc,n$ gives
	\begin{align*}
		\beta^{n,h}(k+1) & \le (1-h\kappa_h)\beta^{n,h}(k) + Ch,
	\end{align*}
	where $\kappa_h := c_0 - 2K_b - Ch$.
	From \eqref{eq:kappa} we can choose $h_0 > 0$ such that $\inf_{h \in (0,h_0)} \kappa_h > 0$ and $1-h\kappa_h \in (0,1)$ for all $h \in (0,h_0)$.
	Then for all $h \in (0,h_0)$, 
	\begin{align*}
		\beta^{n,h}(k+1) & \le (1-h\kappa_h)^2\beta^{n,h}(k-1) + (1-h\kappa_h)Ch + Ch \le \dotsb \\
		& \le (1-h\kappa_h)^{k+1}\beta^{n,h}(0) + \sum_{j=0}^k (1-h\kappa_h)^jCh \\
		& \le (1-h\kappa_h)^{k+1}C + \frac{Ch}{1-(1-h\kappa_h)} \le C.
	\end{align*}	
	Applying this back to \eqref{eq:finite-moment-uniform-Euler-2} gives
	\begin{align*}
		\beta_i^{n,h}(k+1) & \le (1-h\kappa_h)\beta_i^{n,h}(k) + Ch,
	\end{align*}
	which again gives $\beta_i^{n,h}(k+1) \le C$ and verifies \eqref{eq:finite-moment-uniform-Euler-1}.
	
	Using \eqref{eq:finite-moment-uniform-Euler-1} and Remark \ref{rmk:linear-growth}(b), we immediately have the first statement, which further implies the second statement.
	This completes the proof.
%	$$\sup_{h \in (0,h_0)} \sup_{n \ge 1} \max_{i=1,\dotsc,n} \sup_{s \ge 0} \Emb |X_i^{n,h}(s)-X_i^{n,h}(s_h)|^2 < \infty.$$
%	In fact, we have
%	$$\sup_{h \in (0,h_0)} \sup_{n \ge 1} \max_{i=1,\dotsc,n} \Emb |X_i^{n,h}(s)-X_i^{n,h}(s_h)|^2 \le C(s-s_h) \le Ch.$$
%	Combining this and \eqref{eq:finite-moment-uniform-Euler-1} completes the proof.	
\end{proof}

\begin{proof}[Proof of Theorem \ref{thm:moment-convergence-uniform-Euler}]
	Recall $h_0$ in Lemma \ref{lem:finite-moment-uniform-Euler}.	
	Using It\^o's formula, we have
	\begin{align*}
		\Emb |X_i^{n,h}(t)-X_i^n(t)|^2
		& = \Emb \int_0^t 2 \left( X_i^{n,h}(s)-X_i^n(s) \right) \cdot \left( f(X_i^{n,h}(s_h)) - f(X_i^n(s)) \right. \\
		& \qquad \left. + \frac{1}{n} \sum_{j=1}^n \xi_{ij}^n b(X_i^{n,h}(s_h),X_j^{n,h}(s_h)) - \frac{1}{n} \sum_{j=1}^n \xi_{ij}^n b(X_i^n(s),X_j^n(s)) \right) ds.
	\end{align*}
	This implies that the functions $$\gamma_i^n(t) := \Emb |X_i^{n,h}(t)-X_i^n(t)|^2, \quad \gamma^n(t) := \frac{1}{n} \sum_{i=1}^n \Emb |X_i^{n,h}(t)-X_i^n(t)|^2$$ are differentiable.
	
	By adding and subtracting terms, we have
	\begin{align*}
		& \Emb \left[ \left( X_i^{n,h}(s)-X_i^n(s) \right) \cdot \left( f(X_i^{n,h}(s_h)) - f(X_i^n(s)) \right) \right] \\
		& \le \Emb \left[ \left( X_i^{n,h}(s)-X_i^n(s) \right) \cdot \left( f(X_i^{n,h}(s_h)) - f(X_i^{n,h}(s)) \right) \right] \\
		& \quad + \Emb \left[ \left( X_i^{n,h}(s)-X_i^n(s) \right) \cdot \left( f(X_i^{n,h}(s)) - f(X_i^n(s)) \right) \right] \\
		& \le C \sqrt{\gamma_i^n(s) h} - c_0 \gamma_i^n(s),
	\end{align*}
	where the last line uses the Cauchy-Schwarz inequality, the Lipschitz property of $f$ and Lemma \ref{lem:finite-moment-uniform-Euler} for the first term and \eqref{eq:c0} for the second term.
	Also by adding and subtracting terms, we have 
	\begin{align*}
		& \Emb \left[ \left( X_i^{n,h}(s)-X_i^n(s) \right) \cdot \left( \frac{1}{n} \sum_{j=1}^n \xi_{ij}^n b(X_i^{n,h}(s_h),X_j^{n,h}(s_h)) - \frac{1}{n} \sum_{j=1}^n \xi_{ij}^n b(X_i^n(s),X_j^n(s)) \right) \right] \\
		& = \Emb \left[ \left( X_i^{n,h}(s)-X_i^n(s) \right) \cdot \frac{1}{n} \sum_{j=1}^n \xi_{ij}^n \left( b(X_i^{n,h}(s_h),X_j^{n,h}(s_h)) - b(X_i^{n,h}(s),X_j^{n,h}(s)) \right) \right] \\
		& \quad + \Emb \left[ \left( X_i^{n,h}(s)-X_i^n(s) \right) \cdot \frac{1}{n} \sum_{j=1}^n \xi_{ij}^n \left( b(X_i^{n,h}(s),X_j^{n,h}(s)) - b(X_i^n(s),X_j^n(s)) \right) \right] \\
		& \le C\sqrt{\gamma_i^n(s)h}  + K_b \gamma_i^n(s) + K_b \sqrt{\gamma_i^n(s)\gamma^n(s)},
%		& \le \frac{c_0-2K_b}{2}\gamma_i^n(s) + Ch + K_b \gamma_i^n(s) + \frac{K_b}{2} \gamma_i^n(s) + \frac{K_b}{2} \gamma^n(s).
	\end{align*}
	where the last line uses the Cauchy-Schwarz inequality, the Lipschitz property of $b$ and Lemma \ref{lem:finite-moment-uniform-Euler}.
	Combining these two estimates gives
	\begin{equation}
		\label{eq:moment-convergence-uniform-Euler-1}
		\gamma_i^n(t) - \gamma_i^n(r) \le -2(c_0-K_b) \int_r^t \gamma_i^n(s) \,ds + C\sqrt{h} \int_r^t \sqrt{\gamma_i^n(s)} \,ds + 2K_b \int_r^t \sqrt{\gamma_i^n(s)\gamma^n(s)} \,ds
	\end{equation}
	for all $t > r \ge 0$.
	Taking the average over $i=1,\dotsc,n$, we get
	\begin{equation*}
		\gamma^n(t) - \gamma^n(r) \le -2(c_0-2K_b) \int_r^t \gamma^n(s) \,ds + C\sqrt{h} \int_r^t \sqrt{\gamma^n(s)} \,ds.
	\end{equation*}
%	and taking the time derivative, we have
%	\begin{align*}
%		\frac{d}{dt} \gamma^n(t) \le - (c_0 - 2K_b) \gamma^n(t) + Ch.
%	\end{align*}
	Since the function $\gamma^n(t)$ is non-negative and differentiable with $\gamma^n(0)=0$, using Lemma \ref{lem:ODE} (with $a_1=2(c_0-2K_b), a_2=C\sqrt{h}, a_3=0$) we have
%	\begin{align*}
		$\gamma^n(t) \le Ch$.
%	\end{align*}
	Applying this to \eqref{eq:moment-convergence-uniform-Euler-1} gives
	\begin{equation*}
		\gamma_i^n(t) - \gamma_i^n(r) \le -2(c_0-K_b) \int_r^t \gamma_i^n(s) \,ds + C\sqrt{h} \int_r^t \sqrt{\gamma_i^n(s)} \,ds.
	\end{equation*}
	Since the function $\gamma_i^n(t)$ is non-negative and differentiable with $\gamma_i^n(0)=0$, it follows from Lemma \ref{lem:ODE} again that $\gamma_i^n(t) \le Ch$, uniformly in $h \in (0,h_0)$, $t \ge 0$, $n \in \Nmb$ and $i=1,\dotsc,n$.
	This completes the proof.
%	\begin{align*}
%		\frac{d}{dt} \gamma_i^n(t) \le - (c_0 - K_b) \gamma_i^n(s) + Ch.
%	\end{align*}	
%	Since $\gamma_i^n(0)=0$, we have the desired result.
\end{proof}

\appendix

\section{A Wasserstein distance result}

In this section we prove Lemma \ref{lem:Wasserstein-empirical} on the Wasserstein distance about the empirical measure of independent (but not necessarily identically distributed) random variables.
It is a natural generalization of \cite[Theorem 1]{FournierGuillin2015rate} where i.i.d.\ samples are studied.
It is also worth mentioning that for i.i.d.\ samples, the upper bounds are obtained in \cite[Lemma 3.7 and Appendix]{GuoObloj2019computational} for complete cases with explicit constants that was not provided in \cite[Theorem 1]{FournierGuillin2015rate}.
But the three cases in Lemma \ref{lem:Wasserstein-empirical} below are sufficient for our use and we provide a proof for completeness.

\begin{Lemma}
	\label{lem:Wasserstein-empirical}
	Let $\{Y_i : i \in \Nmb\}$ be independent $\Rmb^d$-valued random variables.
	Write $$\mubar_i := \Lmc(Y_i), \quad \nu_n := \frac{1}{n} \sum_{i=1}^n \delta_{Y_i}, \quad \nubar_n := \frac{1}{n} \sum_{i=1}^n \mubar_i.$$	
	Let $p>0$.
	Assume that $\sup_{i \in \Nmb} \Emb|Y_i|^q < \infty$ for some $q > p$.
	Then there exists a constant $C$ depending only on $p,q,d$ such that, for all $n \ge 1$,
	\begin{align*}
		\Emb W_p^p(\nu_n,\nubar_n) & \le C \left( \int_{\Rmb^d} |x|^q \, \nubar_n(dx) \right)^{p/q} \\
		& \times \begin{cases}
		n^{-1/2} + n^{-(q-p)/q} & \mbox{if } p>d/2 \mbox{ and } q \ne 2p, \\
		n^{-1/2} \log(1+n) + n^{-(q-p)/q} & \mbox{if } p=d/2 \mbox{ and } q \ne 2p, \\
		n^{-p/d} + n^{-(q-p)/q} & \mbox{if } p \in (0,d/2) \mbox{ and } q \ne d/(d-p).
		\end{cases}
	\end{align*}
\end{Lemma}

\begin{proof}[Proof of Lemma \ref{lem:Wasserstein-empirical}]
	Fix $A \subset \Rmb^d$.
	In view of the proof of \cite[Theorem 1]{FournierGuillin2015rate},  
	it suffices to verify that
	\begin{equation}
		\label{eq:Wasserstein-claim}
		\Emb |\nu_n(A)-\nubar_n(A)| \le \min \left\{ 2\nubar_n(A), \sqrt{\nubar_n(A)/n} \right\}.
	\end{equation} 
	For this, clearly we have
	$$\Emb |\nu_n(A)-\nubar_n(A)| \le \Emb \nu_n(A) + \nubar_n(A) = 2\nubar_n(A).$$
	Also note that, by the independence of $\{Y_i:i\in\Nmb\}$,
	\begin{align*}
		\Emb |\nu_n(A)-\nubar_n(A)|^2 & = \Emb \left[ \frac{1}{n} \sum_{i=1}^n \left( \one_{\{Y_i \in A\}} - \mubar_i(A) \right) \right]^2 
		= \frac{1}{n^2} \sum_{i=1}^n \Emb \left[ \one_{\{Y_i \in A\}} - \mubar_i(A) \right]^2 \\
		& = \frac{1}{n^2} \sum_{i=1}^n \mubar_i(A) (1-\mubar_i(A)) \le \frac{1}{n^2} \sum_{i=1}^n \mubar_i(A) = \frac{1}{n}\nubar_n(A).
	\end{align*}
	This completes the proof.
\end{proof}

%\bibliographystyle{plain}
%\bibliography{Stationarybib}


\end{document}